\documentclass[10pt]{article}

\usepackage[utf8]{inputenc}
\usepackage{a4wide}
\usepackage{color}
\usepackage{amsfonts,amsmath,amsthm,amssymb}
\usepackage{verbatim}
\usepackage{epsfig}
\usepackage{graphics}
\usepackage{mathtools}
\usepackage{array}
\usepackage{enumerate}
\usepackage[ruled]{algorithm}
\usepackage{algorithmic}
\usepackage{bbm}
\usepackage[font=small,labelfont=bf]{caption}
\usepackage{wrapfig}
\usepackage{subcaption}
\usepackage[all, knot, cmtip]{xy} 
\usepackage{wrapfig}

\usepackage{xcolor}
\usepackage[pdfpagemode=UseNone,bookmarksopen=false,colorlinks=true,urlcolor=blue,citecolor=blue,citebordercolor=blue,linkcolor=blue]{hyperref}

\usepackage{mathrsfs}

\usepackage[top=2.5cm,left=2.5cm,right=2.5cm,bottom=2.5cm]{geometry}

\newcommand{\ndN}{\mathbb{N}}
\newcommand{\ndZ}{\mathbb{Z}}

\newcommand{\ndR}{\mathbb{R}}

\renewcommand{\Pr}[1]{\mathbb{P}(#1)}
\newcommand{\Prb}[1]{\mathbb{P}\left( #1 \right)}
\newcommand{\Ex}[1]{\mathbb{E}[#1]}

\newcommand{\Va}[1]{\mathbb{V}[#1]}

\newcommand{\w}{w}

\newcommand{\cO}{\mathcal{O}}

\newcommand{\cF}{\mathcal{F}}
\newcommand{\cB}{\mathcal{B}}

\newcommand{\cN}{\mathcal{N}}

\newcommand{\cT}{\mathcal{T}}

\newcommand{\cA}{\mathcal{A}}

\newcommand{\scO}{\mathscr{O}}

\newcommand{\scS}{\mathscr{S}}
\newcommand{\scH}{\mathscr{H}}
\newcommand{\scD}{\mathscr{D}}
\newcommand{\scX}{\mathscr{X}}
\newcommand{\scG}{\mathscr{G}}
\newcommand{\scF}{\mathscr{F}}

\newcommand{\scA}{\mathscr{A}}
\newcommand{\scB}{\mathscr{B}}
\newcommand{\scC}{\mathscr{C}}
\newcommand{\scP}{\mathscr{P}}

\newcommand{\mA}{\mathsf{A}}
\newcommand{\mC}{\mathsf{C}}
\newcommand{\mD}{\mathsf{D}}
\newcommand{\mB}{\mathsf{B}}

\newcommand{\mG}{\mathsf{G}}

\newcommand{\mH}{\mathsf{H}}

\newcommand{\mP}{\mathsf{P}}

\newcommand{\mX}{\mathsf{X}}
\newcommand{\mY}{\mathsf{Y}}

\newcommand{\me}{\mathsf{e}}

\newcommand{\CRT}{\mathcal{T}_\me}

\newcommand{\Sym}{\mathrm{Sym}}
\newcommand{\RSym}{\mathrm{RSym}}

\newcommand{\cE}{\mathcal{E}}

\newcommand{\Di}{\textnormal{D}}

\newcommand{\eqdist}{\,{\buildrel d \over =}\,}

\newcommand{\convdis}{\,{\buildrel w \over \longrightarrow}\,}
\newcommand{\convp}{\,{\buildrel p \over \longrightarrow}\,}

\newcommand{\Set}{\textsc{SET}}

\newtheorem{theorem}{Theorem}[section]

\newtheorem{corollary}[theorem]{Corollary}
\newtheorem{proposition}[theorem]{Proposition}
\newtheorem{lemma}[theorem]{Lemma}

\numberwithin{equation}{section}

\title{\textbf{Asymptotic properties of random unlabelled block-weighted graphs}}
\date{}

\author{Benedikt Stufler\thanks{University of Zurich, E-mail: benedikt.stufler@math.uzh.ch; The author gratefully acknowledges support by the German Research Foundation DFG, STU 679/1-1}}

\begin{document}
	
	\maketitle

\begin{abstract}
	We study the asymptotic shape of random  unlabelled graphs subject to certain subcriticality conditions. The graphs are sampled with probability proportional to a product of Boltzmann weights assigned to their $2$-connected components. As their number of vertices tends to infinity, we show that they admit the Brownian tree as Gromov--Hausdorff--Prokhorov scaling limit, and converge in a strengthened Benjamini--Schramm sense toward an infinite random graph. We also consider a family of random graphs that are allowed to be disconnected. Here a giant connected component emerges and the small fragments converge without any rescaling towards a finite random limit graph.  %Our approach is a symbioses of stochastic process methods and the block-decomposition of cycle-pointed graphs. %that also facilitates a transfer of a wide range of asymptotic properties concerning additive and extremal graph parameters from rooted graphs to unrooted graphs, such as central limit theorems for the number of blocks, cut-vertices, or  vertices with a given degree.
\end{abstract}

\let\thefootnote\relax\footnotetext{ \\\emph{MSC2010 subject classifications}. Primary 60C05; secondary 05C80. }

\vspace {-0.5cm}

\section{Introduction and main results}
The probabilistic study of random graphs from restricted classes  has received some attention in recent literature~\cite{2017arXiv170702144G,2016arXiv160608769G,MR3184197,MR2350456,MR2873207,2016arXiv160903974C,2015arXiv150406344C,2015arXiv151208889D}. In the present work, we finalize a project \cite{inannals, 2016arXiv161202580S,2015arXiv150402006S,2015arXiv150207180P,2014arXiv1412.6333S,Mreplaceme} that showed how stochastic process methods combined with Boltzmann sampling principles based on combinatorial bijections may fruitfully be applied in this context. Previous efforts focused on random labelled graphs and unlabelled rooted graphs. We complete the picture by providing probabilistic graph limits concerning  graphs that are \emph{unlabelled} and \emph{unrooted}.

There are various kinds of graph limits. Gromov--Hausdorff-Prokhorov scaling limits describe the asymptotic global geometric shape of random discrete objects, with the archetype of limit objects being given by Aldous' Brownian continuum random tree  (CRT) \cite{MR1085326,MR1166406,MR1207226}.  Since Aldous' pioneering work, the universality class of the CRT received considerable attention,  see for example Caraceni~\cite{Ca}, Curien, Haas and Kortchemski \cite{MR3382675}, and Janson and Stef\'ansson~\cite{MR3342658}. While scaling limits are concerned with the global behaviour of a sequence of random metric spaces, they contain little information on asymptotic local properties. These are better described by distributional limits \cite{MR1873300}. Convergence of a sequence of random rooted objects in this sense  boils down to convergence  of finite neighbourhoods of the root vertices. See for example Aldous and Pitman \cite{MR1641670}, Janson \cite{MR2908619} and references given therein, Bj\"ornberg and Stef\'ansson \cite{MR3183575}, and Stephenson \cite{2014arXiv1412.6911S}. These notions are usually defined for random connected graphs. In certain contexts, it is also interesting to allow disconnected graphs. The emergence of a giant connected component together with a almost surely finite limit graph for the small fragments has been observed for a variety of models of random graphs. See  in particular the work by McDiarmid \cite{MR3084594} and references given therein.

Although considerable progress has been made regarding random ordered structures or labelled graphs, less is known about complex structures considered up to symmetry, in particular unlabelled graphs. Recall that for any class of graphs  one may consider three random models: “each labelled graph with $n$ vertices equally likely”, “each unlabelled rooted graph with $n$ vertices equally likely”, and each “unlabelled graph with $n$ vertices equally likely”. In the most basic case, the class of trees, all three models have been well understood. For example, the scaling limit for the labelled case was established by Aldous in his pioneering work \cite{MR1085326}, the scaling limit for unlabelled rooted trees was studied Marckert and Miermont \cite{MR2829313}, Haas and Miermont \cite{MR3050512} and Panagiotou and S. \cite{Panagiotou2017}, and the unlabelled case (without root vertices) was treated in S.~\cite{2014arXiv1412.6333S}. See also  Wang~\cite{2016arXiv160408287W}, who established scaling limits for a family of unrooted weighted plane trees indexed by their height. It is natural to aim for similar results for more complex classes of graphs. In the labelled case, a scaling limit for subcritical random graphs was established by Panagiotou, S. and Weller \cite{inannals} and  Benjamini--Schramm limits are given in S. \cite{2016arXiv161202580S} and Georgakopolous and Wagner~\cite{2017arXiv170100973G}. These results were preceded by the study of many combinatorial parameters as in Bernasconi, Panagiotou and Steger \cite{MR2534261, MR2789731}, and Drmota and Noy \cite{MR3184197}. 
The unlabelled rooted case was studied by Drmota, Fusy, Kang, Kraus and Ru\'e \cite{MR2873207}, who conducted a combinatorial study of additive graph parameters. S. \cite{2015arXiv150402006S} used a probabilistic approach to treat extremal properties and establish limits encompassing a  scaling limit, a local weak limit for the vicinity of the root vertex, and a Benjamini--Schramm limit. Georgakopolous and Wagner~\cite{2017arXiv170100973G} showed local weak convergence for the vicinity of the fixed root by an analytic approach. For unlabelled unrooted graphs, less is known, with a notable exception by Kraus \cite{MR2735357}, who studies the degree distribution of dissections of polygons considered up to symmetry, and \cite{2017arXiv170100973G}, where a Benjamini--Schramm limit for random unlabelled unrooted graphs from subcritical graph classes was established. However, this does not answer how these graphs behave asymptotically on a global scale. For this reason, the present paper aims to complete the picture by providing a Gromov--Hausdorff--Prokhorov limit. As a  byproduct, this yields precise limits for extremal properties such as the diameter, but also for the distances between $k$ independently sampled uniform points in the graph.  Rather than studying the classical model of uniform random graphs from such classes, we formulate our results for random graphs sampled according to Boltzmann weights on the $2$-connected components. Weighted graphs have also been studied recently in other contexts, see in particular the works by McDiarmid~\cite{MR3084594} and Richier~\cite{2017arXiv170401950R}, as well as references given therein. The reason for this higher level of generality is twofold. First, it allows us to gain a clearer perspective on the phenomenon under consideration. For example, any property of the uniform labelled $n$-vertex tree has an analogue in the more general probabilistic context of critical Galton--Watson trees with a reasonably well-behaved reproduction law. Second, further studies of different weight sequences may lead to the discovery of new phenomena. Contemporary examples include the   $\alpha$-stable maps by Le Gall and Miermont~\cite{MR2778796} and limit theorems for face-weighted outerplanar maps  in~\cite{2017arXiv171004460O}.

Suppose that for each $2$-connected graph $B$ (including the complete graph $K_2$ with two vertices) we are given a weight $\iota(B) \ge 0$ such that isomorphic graphs receive the same weight. To any connected graph $C$ we may then assign the weight 
\begin{align}
\label{eq:thweight}
\omega(C) = \prod_{B} \iota(B)
\end{align}
with the index $B$ ranging over all $2$-connected components of $C$, that is, maximal $2$-connected subgraphs. If $C$ consists of a single vertex, then it receives weight $1$.  We may then consider the random connected unlabelled graph $\mC_n^\omega$ with $n$ vertices, sampled with probability proportional to its $\omega$-weight, and likewise the unlabelled rooted random graph $\mA_n^\omega$. This model encompasses so called random unlabelled graphs from block-stable classes, which correspond to the special case where each $\iota$-weight is required to be either equal to $1$ or $0$. Block-stable classes of graphs have received some attention in recent literature,   see for example McDiarmid and Scott \cite{2014arXiv1408.4257M}.

The direct study of $\mC_n^\omega$ is challenging, as the structure of the symmetries of objects without roots is much more complex as in the rooted case, where each symmetry is required to fix the root vertex. So, instead of directly studying unrooted unlabelled graphs, we are going to take a more economic approach and geometrically approximate $\mC_n^\omega$ by a random rooted graph having size $n + O_p(1)$. More precisely, we are going to construct a random rooted graph $\mD_n$ with size $d_n = |\mD_n| = O_p(1)$ such that the graph $\mD_n + \mA_{n - d_n}^\omega$ obtained by identifying the root of $\mD_n$ with the root of $\mA_{n - h_n}^\omega$ approximates $\mC_n^\omega$ in total variation. 

In order for the random graph $\mC_n^\omega$ to behave in a tree-like manner, we will make an assumption on the weight-sequence, which generalizes the definition of subcriticality for block-stable classes of unlabelled graphs given in \cite[Sec. 5]{MR2873207}: Define the cycle index sum
\[
Z_{(\scB')^\iota}(s_1, s_2, \ldots) = \sum_{k \ge 1} \sum_{B' \in \scB'_k} \frac{\iota(B')}{|k|!} \sum_{\sigma} s_1^{\sigma_1} s_2^{\sigma_2} \cdots s_k^{\sigma_k},
\]
with $\scB_k'$ denoting the set of all $2$-connected graphs with vertex set $\{*,1, \ldots, k\}$ for arbitrary $k \ge 1$, the sum index $\sigma:[k] \to [k]$ ranging over the elements of the permutation group of order $k$ such that the canonically extension $\bar{\sigma}$ with $\bar{\sigma}|_{[k|]} = \sigma$ and $\bar{\sigma}(*) = *$ is an automorphism of $B'$, and $\sigma_i$ denoting the number of cycles of length $i$ in $\sigma$. Likewise, we define the cycle sum
\[
Z_{(\scC)^\omega}(s_1, s_2, \ldots) = \sum_{k \ge 1} \sum_{C \in \scC_k}  \frac{\omega(C)}{|k|!} \sum_{\sigma} s_1^{\sigma_1} s_2^{\sigma_2} \cdots s_k^{\sigma_k},
\]
with $\scC_k$ denoting the set of all connected graphs with labels in $[k]$ and the sum index $\sigma$ ranging over automorphisms of $C$. Furthermore, we set for each $i \ge 1$
\[
\tilde{\scA}^{\omega^i}(z) = \sum_A \omega(A)^i z^{|A|}
\]
with the sum index $A$ ranging over all rooted unlabelled graphs. We require that the radius of convergence $\rho_\scA$ of $\scA^\omega(z)$ and  the bivariate sum
\[
g(x,y) = \exp\left( Z_{(\scB')^\iota}(x, \tilde{\scA}^{\omega^2}(y^2), \tilde{\scA}^{\omega^3}(y^3), \ldots) + \sum_{i \ge 2} \frac{1}{i} Z_{(\scB')^\iota}(\tilde{\scA}^{\omega^i}(y^i), \tilde{\scA}^{\omega^{2i}}(y^{2i}),  \ldots) \right)
\]
satisfy
\begin{align}
\label{eq:H}
	\rho_\scA>0, \qquad g(\tilde{\scA}^\omega(\rho_\scA) + \epsilon, \rho_\scA + \epsilon) < \infty, \qquad \text{and} \qquad Z_{\scC^\omega}(0, ( \rho_\scA + \epsilon)^2, ( \rho_\scA + \epsilon)^3, \ldots) < \infty
\end{align}
for some $\epsilon>0$.

Although this requirement seems rather abstract, it is known to be satisfied for a wide range of random graphs that appear naturally in combinatorics. For example, condition~$\eqref{eq:H}$ holds if $\mC_n^\omega$ is the uniform random connected unlabelled series-parallel graph, cacti graph or outerplanar graph  with $n$ vertices. Or, more generally, this encompasses so called random graphs from subcritical classes of unlabelled graphs, which includes random graphs from classes defined by a finite set of $3$-connected components. See Section~6 of the work~\cite{MR2873207} by Drmota, Fusy, Kang, Kraus, and Ru\'e for details.

\begin{theorem}
	\label{te:approx}
	Suppose that the tree-like requirement \eqref{eq:H} is satisfied.
	Then there is a coupling of $\mC_n^\omega$ with a random rooted graph $\mD_n$ with size $d_n = |\mD_n| = O_p(1)$, and the random rooted graph $\mA_{n - h_n}^\omega$, such that the graph $\mD_n + \mA_{n - h_n}^\omega$ obtained by identifying the root of $\mD_n$ with the root of $\mA_{n - h_n}^\omega$ approximates $\mC_n^\omega$ in total variation. The speed of convergence is exponential, that is,
	\begin{align}
		\label{eq:tv}
		d_{\textsc{TV}}(\mC_n^\omega, \mD_n + \mA^\omega_{n - d_n}) \le C \exp(-cn)
	\end{align}
	for some constants $C,c>0$ that do not depend on $n$.
\end{theorem}

%\textcolor{red}{add tail-bound for the size of attached graph!}

We obtain Theorem~\ref{te:approx} by combining Gibbs partition methods~\cite{2016arXiv161001401S} with the cycle pointing technique developed by Bodirsky, Fusy, Kang and Vigerske \cite{MR2810913}. The idea to use cycle pointing to this end was also used in  S. \cite{2014arXiv1412.6333S} for the probabilistic study of unlabelled trees with possible degree restrictions. The present work and \cite{2014arXiv1412.6333S} intersect precisely for the model of uniform random unlabelled trees {\em without} degree restrictions. Unlabelled trees with proper vertex degree restrictions do not fall into the family of random graphs we consider here, and also require a different type of decomposition. % and as the model of random graphs considered here does not admit degree restrictions. But this allows us to use decomposition that is based on the blocks of the graph.

Theorem~\ref{te:approx} allows us to transfer a large class of asymptotic graph properties from $\mA_n^\omega$ to $\mC_n^\omega$. Note that this approach does not work as well the other way. We may at best deduce that an asymptotic property of $\mC_n^\omega$ must also hold for the random rooted graph $\mA_{n - h_n}^\omega$ which has a {\em random} size, but this does a priori not imply that the property also asymptotically holds for $\mA_n^\omega$ which has a {\em deterministic} size. %Hence it is more natural and economic to study $\mA_n^\omega$ and then transfer the results to~$\mC_n^\omega$.  

We now state our main applications.

\begin{theorem}
	\label{te:scalinglimit}
	Suppose that \eqref{eq:H} holds, and let $\mu_n$ denote the uniform measure on the vertices of $\mC_n^\omega$. Then there is a constant $c_\omega>0$ such that
	\begin{align}
	\label{eq:limt}
	 \left(\mC_n^\omega, \frac{c_\omega}{\sqrt{n}} d_{\mC_n^\omega}, \mu_n \right) \convdis (\CRT, d_{\CRT}, \mu)
	 \end{align}
	 in the Gromov--Hausdorff--Prokhorov sense, with $(\CRT, d_{\CRT}, \mu)$ denoting the Brownian continuum random tree. Moreover, there are constants $C,c>0$ such the diameter $\Di(\mC_n^\omega)$ satisfies for all $n$ the tail bound
	 \begin{align}
	 \label{eq:tailb}
	 \Pr{\Di(\mC_n^\omega) \ge x} \le C \exp(-cx^2 / n).
	 \end{align}
\end{theorem}

The idea behind the scaling limit is that the graph $\mD_n$ contracts to a single point when rescaled by $n^{-1/2}$, and hence the Gromov--Hausdorff--Prokhorov distance between the rescaled versions of $\mA_{n - h_n}^\omega$ and  $\mC_n^\omega$ tends in probability to zero. Hence we may build upon previous Gromov--Hausdorff limits in the rooted case~\cite[Thm. 6.14]{2015arXiv150402006S}, which we extend in a non-trivial way to obtain convergence in the rooted  Gromov--Hausdorff--Prokhorov sense. This form of limit yields precise asymptotic expressions for the diameter of $\mC_n^\omega$ and distances between $k$ independently sampled random points, see for example~\cite[Prop. 10]{MR2571957} for a justification in a more general context. For example, it follows from~Theorem~\ref{te:scalinglimit} that the graph distance $d_{\mC_n^\omega}(v^1, v^2)$ of two independently and uniformly selected points $v^1, v^2 \in \mC_n^\omega$ satisfies \begin{align}2 c_{\omega} d_{\mC_n^\omega}(v^1, v^2)/ \sqrt{n} \convdis \text{Rayleigh}(1)\end{align} for a Rayleigh-distributed limit, given by its probability density $x\exp(-x^2/2)$. Note that the tail-bound ~\eqref{eq:tailb} implies that the distance $d_{\mC_n^\omega}(v^1, v^2)$ is $p$-uniformly integrable for any $p \in \ndN$, yielding
\begin{align}
	\Ex{  d_{\mC_n^\omega}(v^1, v^2)^p } \sim  n^{p/2} 2^{-p/2} c_\omega^{-p}\Gamma(1 + p/2).
\end{align}
The Gromov--Hausdorff--Prokhorov universality class of the Brownian continuum random tree (and other continuous limit objects) was also studied in a recent work on Voronoi tesselations~\cite{soda_accepted}, and the results given there also apply to the random graph $\mC_n^\omega$ by Theorem~\ref{te:scalinglimit}.

Among the many classes of graphs that satisfy the tree-like assumption~\eqref{eq:H}, uniform unlabelled outerplanar graphs have received particular attention in  \cite{MR2350456}. This corresponds to the case where we assign weight $1$ to each graph that may be drawn in the plane such that no edges intersect and each vertex lies on the frontier of the outer face. All other graphs receive weight $0$. We  derive a numeric approximation of the scaling factor for this class of graphs.

\begin{proposition}
	\label{pro:proconst}
The scaling constant $c_{\omega_\scO}$ of the class of unlabelled outerplanar graphs is approximately given by \[c_{\omega_\scO} \approx 0.9864689.\]
\end{proposition}
 This differs from the case of random labelled outerplanar graphs for which the constant is approximately given by $\approx 0.960$ \cite[Prop. 8.6]{inannals} and the case of random outerplanar maps for which it equals $9 / (7 \sqrt{2}) \approx 0.909$ \cite{caraceni2016,stufler2017}.

\begin{theorem}
	\label{te:bslimit}
	Suppose that \eqref{eq:H} holds. Then there is a locally finite limit graph $\hat{\mC}^\omega$ with a distinguished vertex $\hat{v} \in \hat{\mC}^\omega$ such that $\mC_n^\omega$ converges toward $(\hat{\mC}^\omega, \hat{v})$ in the Benjamini--Schramm sense. Even stronger, if $v_n$ denotes uniformly at random drawn vertex from $\mC_n^\omega$, and $k_n = o(\sqrt{n})$ is a fixed deterministic sequence of non-negative integers, then 
	\begin{align}
	\lim_{n \to \infty} d_{\textsc{TV}}( U_{k_n}(\mC_n^\omega, v_n), U_{k_n}(\hat{\mC}^\omega, \hat{v})) = 0,
	\end{align}
	with $U_{k_n}(\cdot, \cdot)$ denoting the subgraph induced by all vertices with distance at most $k_n$ from the specified vertex.
\end{theorem}

The Benjamini--Schramm convergence is deduced by observing that the  $k_n$-neighbourhood of a uniformly at random drawn vertex $v_n \in \mH_n + \mA_{n - h_n}^\omega$ lies with high probability entirely in $\mA_{n - h_n}^\omega$ and hence it suffices to establish local convergence for $\mA_{n - h_n}^\omega$. This is achieved by using the stronger form of  Benjamini--Schramm convergence for $\mA_n^\omega$ established in S. \cite{2015arXiv150402006S}. As a byproduct, we obtain that the Benjamini--Schramm limits of $\mA_n^\omega$ and $\mC_n^\omega$ agree. For the case where $\iota(B) \in \{0,1\}$ for all $B$ and for which the $\omega$-weight of all trees is positive, we hence recover and extend the Benjamini--Schramm limit of uniform random unlabelled unrooted graphs from subcritical classes that was  established by Georgakopoulos and Wagner \cite[Thm. 4.4]{2015arXiv151203572G} and corresponds to convergence of neighbourhoods with  constant radius instead of $o(\sqrt{n})$. We remark that the almost sure convergence stated in \cite[Thm. 4.5]{2015arXiv151203572G} may also be deduced from Theorem~\ref{te:bslimit} using Skorokhod's representation theorem. In detail: Benjamini--Schramm convergence of the sequence $\mC_n^\omega$ corresponds to weak convergence of the rooted graphs $(\mC_n^\omega, v_n)$ interpreted as random points in the Polish space $(\mathbb{B}, d_{\text{loc}})$ of connected rooted unlabelled locally finite graphs, with the metric $d_{\text{loc}}$ given in Equation~\eqref{eq:benjamini} below. Stating that the sequence $(\mC_n^\omega,v_n)$ converges almost surely is syntactically incorrect, unless we construct the limit and individual graphs on the same probability space. By Skorokhod's representation theorem~\cite[Thm. 3.3]{MR0310933} and the weak convergence $(\mC_n^\omega, v_n) \convdis (\hat{\mC}^\omega, \hat{v})$ it follows that there exist identically distributed copies $\mC_n^* \eqdist (\mC_n^\omega, v_n)$, $n \ge 1$, and $\mC^* \eqdist (\hat{\mC}, \hat{v})$ that are defined on a common probability space and satisfy $\mC_n^* \to \mC^*$ almost surely. 

An advantage of the approach taken in this project is that we obtain precise descriptions of~\emph{both} the asymptotic local and global geometric shape. In order to prevent and rectify certain misconceptions, we emphasize that the scaling limit of random \emph{labelled} graphs from subcritical classes by Panagiotou, S., and Weller~\cite{inannals} does not encompass and is not encompassed by the results of the present paper on random~\emph{unlabelled} graphs. Selecting a random labelled graph  uniformly at random from a class of graphs that is closed under relabelling corresponds to placing a bias that is inversely proportional to the size of its automorphism group, and random unlabelled graphs do not exhibit this bias. The unreasonable yet prevalent misconception that graphs in this context typically have no symmetries is easily rebuked by noting the change of growth constants in the labelled and unlabelled setting~\cite{MR2873207}.

Our last result is an observation on random graphs that are not necessarily connected. That is, random unlabelled elements from the class $\scG^\omega$ of $\omega$-weighted graphs, where the $\omega$-weight of such a graph is defined as the product of $\omega$-weights of its connected components. Similarly as in the approximation of unrooted graphs by rooted graphs in Theorem~\ref{te:approx}, the following limit allows us to transfer "practically every" asymptotic property from the connected to the disconnected regime. % A theorem due to Barbour and Granovsky \cite[Thm. 2.3]{MR2121024} gives us precise information on the asymptotic distribution of the component sizes, yielding the following result.

\begin{corollary}
	\label{co:components}
	Let $\mG_n^\omega$ denote the random unlabelled graph with $n$ vertices sampled with probability proportional to the product of $\omega$-weights of its connected components. If requirement \eqref{eq:H} holds, then the largest connected component $\mC_n$ of $\mG_n^\omega$ has size $|\mC_n| = n + O_p(1)$. More precisely, let  $d \ge 1$ denote the unique largest constant such that all finite graphs connected graphs with positive $\omega$-weight have size in the lattice $1 + d\ndZ$. Then for each $0 \le a < d$ there is a finite random graph $\mG_a$ such that
	\begin{align}
	\mG_n^\omega - \mC_n  \convdis \mG_a
	\end{align}
	as $n$ tends to infinity on the lattice  $a + d\ndN$. Here weak convergence is to be understood in the usual sense, that is, of random elements of the countable set of unlabelled finite graphs. The distribution of $\mG_a$ has Boltzmann-type:
	\[
		\Pr{\mG_a = G} = \omega(G) \rho_\scA^{|G|} \left ( \sum_{H, |H| \equiv a-1 \mod d} \omega(H) \rho_\scA^{|H|} \right)^{-1}, \qquad |G| \equiv a-1 \mod d.
	\]
	Here the sum index $H$ ranges over all unlabelled graphs with size in the lattice $a-1 + d \ndZ$.%, and $\rho_\scA$ denotes the constant from Theorem~\ref{te:block}.
\end{corollary}

Compare with a result for the number of connected components for the case of random unlabelled outerplanar graphs in~\cite[Thm. 5.1]{MR2350456} and under a general smoothness condition given in~\cite[Chap. 4, Sec. 6.4]{MR3560309}. It follows  from Corollary~\ref{co:components} that  the limit Theorems \ref{te:scalinglimit} and \ref{te:bslimit} also hold for the largest component of~$\mG_n^\omega$, and, if we permit disconnected graphs in the notion of local weak convergence, we may also consider Theorem~\ref{te:bslimit} as a limit theorem for the random graph $\mG_n^\omega$. 

A  similar result was observed for random \emph{labelled} graphs from small block-stable classes in S. \cite[Thm. 4.2]{Mreplaceme}, which generalized McDiarmid's~\cite{MR2507738} previous results on random graphs from proper minor-closed addable graph classes. It is natural to expect that, similar as in the labelled case, the tree-like condition~\eqref{eq:H} is not required for Corollary~\ref{co:components} to hold, and may be replaced by merely requiring the series $\tilde{\scA}^\omega(z)$ to have positive radius of convergence. A promising natural setting or level of abstraction for pursuing this question appears to be a probabilistic context. Specifically, this conjecture may be affirmed if a  generalization of the result \cite[Lem. 3.3]{Mreplaceme}  to partition functions of certain sesqui-type branching processes is possible, as these sequences describe the coefficients of $\tilde{\scA}^\omega(z)$ as a special case (see \cite{2015arXiv150402006S} for details). These processes are also of interest in their own right, and have received attention in recent literature by Janson, Riordan, and Warnke~\cite{2017arXiv170600283J}.

%\textcolor{red}{add what this generalizes}
%\subsection*{Methodology}
%The methods used in the present work include \textcolor{red}{...}

\section{Probabilistic graph limits}

\subsection{Distributional convergence}
\label{sec:lowe}
Given two connected, rooted, and locally finite graphs $G^\bullet = (G, v_G)$ and $
H^\bullet = (H, v_H)$ we may consider their distance
\begin{align}
\label{eq:benjamini}
d_{\text{loc}}(G^\bullet, H^\bullet) =  2^{-\sup \{k \in \ndN_0 \,\mid\, U_k(G^\bullet) \simeq U_k(H^\bullet) \}}
\end{align}
with  $U_k(G^\bullet) \simeq U_k(H^\bullet)$ denoting isomorphism of rooted graphs. This defines a premetric on the collection of all rooted locally finite connected graphs.  Two such graphs have distance zero, if and only if they are isomorphic. Hence we obtain a metric on the collection $\mathbb{B}$ of all unlabelled, connected, rooted, locally finite graphs.  There are some set-theoretic caveats that actually require us to work with a set of representatives instead of a collection of proper classes, but we may safely ignore this purely notational issue.

Weak convergence of a sequence $(\mG_n, v_n)_{n \ge 1}$ of random pointed graphs in $\mathbb{B}$ is also called \emph{local weak convergence}. In the special case where for each $n$ the random graph $\mG_n$ is almost surely finite, and the root $v_n$ is selected uniformly at random from its vertices, it also called \emph{distributional} or \emph{Benjamini--Schramm} convergence. 

%This notion of convergIf the root degrees additionally satisfy some moment conditions, it readily yields laws of large numbers for additive graph parameters such as the number of spanning trees or subgraph count asymptotics \cite{MR2160416, 2015arXiv150408103K}.

\subsection{Gromov--Hausdorff--Prokhorov convergence}
\label{sec:ghc}

Most parts of the present exposition follow~\cite[Sec. 6]{MR2571957}. Given two compact subsets $K_1, K_2$ of a metric space $(X, d_X)$, we may consider their \emph{Hausdorff distance}
\[
d_{\mathrm{H}}(K_1, K_2) = \inf\{\epsilon >0 \mid K_2 \subset U_\epsilon(K_1), K_1 \subset U_\epsilon(K_2)\}.
\]
Here  $U_\epsilon(K) := \{ x \in X \mid d_X(x,K) < \epsilon\}$ denotes the $\epsilon$-thickening of a subset $K \subset X$.   The \emph{Prokhorov distance} between two Borel probability measures $P_1, P_2$ on  $X$ is defined by
\[
d_{\mathrm{H}}(P_1, P_2) = \inf \{\epsilon >0 \mid \text{ for all $A\subset X$ closed: } P_1(A) \le P_2(U_\epsilon(A)) + \epsilon, P_2(A) \le P_1(U_\epsilon(A)) + \epsilon\}.
\]
Let $(X, d_X, P_X)$, $(Y, d_Y, P_Y)$ be compact metric spaces equipped with Borel probability measures. For any metric space $(E, d_E)$ and isometric embeddings $\iota_X: X \to E$ and $\iota_Y: Y \to E$ we may consider the push-forward measures $P_X \iota_X^{-1}$ and $P_Y \iota_Y^{-1}$. The \emph{Gromov--Hausdorff--Prokhorov} (GHP) distance  between the two  spaces is given by
\[
d_{\mathrm{GHP}}((X, d_X, P_X), (Y, d_Y, P_Y)) = \inf_{(E, d_E), \iota_X, \iota_Y}  \min(d_{\mathrm{H}}(\iota_X(X), \iota_Y(Y)), d_{\mathrm{P}}(P_X \iota_X^{-1}, P_Y \iota_Y^{-1})  ),
\]
with the index ranging over all possible isometric embeddings $\iota_X, \iota_Y$ of $X$ and $Y$ into any possible common metric space $(E, d_E)$.

The GHP distance satisfies the axioms of a premetric on the collection of compact metric spaces equipped with Borel probability measures. The corresponding metric on the quotient space $\mathbb{K}$ is complete and separable. That is, $\mathbb{K}$ is a Polish space. For set-theoretic reasons, we would actually have to work with a set of representatives instead of a collection of proper class, but this a purely notational issues that we may safely ignore.

We are usually not going to distinguish between a measured compact metric space  and the corresponding equivalence class. Also, whenever there is no risk of confusion, we will write $\lambda X$ instead of  $(X, \lambda d_X, P_X)$ for any scalar factor $\lambda >0$ and any compact metric space $(X, d_X)$ equipped with a Borel probability measure~$P_X$.  

If we distinguish points $x_0 \in X$ and $y_0 \in Y$ we may also form the \emph{rooted Gromov--Hausdorff--Prokhorov}-distance $d^c_{\mathrm{GHP}}(X^\bullet, Y^\bullet)$ between the rooted spaces $X^\bullet = (X, x_0)$ and $Y^\bullet = (Y, y_0)$ by
\[
\inf_{(E, d_E), \iota_X, \iota_Y}  \min(d_{\mathrm{H}}(\iota_X(X), \iota_Y(Y)), d_{\mathrm{P}}(P_X \iota_X^{-1}, P_Y \iota_Y^{-1}), d_E(\iota_X(x_0), \iota_Y(y_0)) ).
\]
The rooted GHP-distance $d^c_{\mathrm{GHP}}$ satisfies analogous properties as $d_{\mathrm{GHP}}$, see \cite[Thm. 2.3]{MR3035742} for details.

\section{The block-decomposition of cycle pointed graphs}

In order to deal with the symmetries that complicate the analysis of unlabelled graphs, we will make use of enumerative and probabilistic aspects of the theory of species. In Appendix~\ref{sec:species} we summarize some tools and notions of this theory that we require in the proofs of our main results, and provide further references for a detailed introduction to the topic. A reader with a strong understanding of the symbolic method may skip its lecture and directly proceed with the present section.

%The author's inclination would be to refer the reader to the literature \cite{MR633783,MR2810913} for all details, but this would not be fair towards readers with a more probabilistic background. Hence the present chapter is intended to summarize the combinatorial aspects necessary to understand the proofs of our main results. A reader with a strong understanding of the symbolic method may safely skip this chapter.

\subsection{The block-tree}

A \emph{cut-vertex} of a connected graph is a vertex whose removal disconnects the graph. 
We say a graph is \emph{$2$-connected}, if it is connected, has at least $2$ vertices, but no cut-vertices. This includes the link-graph consisting of two edges joined by an edge.
A \emph{block} $B$ of a graph $G$ is a subgraph that is inclusion maximal with the property of being either an isolated vertex or $2$-connected.  Any two blocks  overlap in at most one vertex. The cut-vertices of a connected graph are precisely the vertices that belong to more than one block. For any connected graph $C$  we may form the associated \emph{block-tree} $T(C)$ that comes with a bipartition of its vertices into two groups of vertices \cite[Ch. 3.1]{MR2744811}. One group corresponds to the blocks of $C$ and the other to its cut-vertices. The edges of the tree $T(C)$ are given by all pairs $\{v,B\}$ with $v$ a cut-vertex and $B$ a block  that contains the vertex $v$.

\subsection{Cycle pointing}

We recall the block-decomposition of cycle-pointed connected graphs given in Bodirsky, Fusy, Kang and Vigerske \cite[Prop. 28]{MR2810913} and check that its compatible with block-weightings. We assume familiarity with the cycle pointing operations, see Appendix~\ref{sec:species} and the references given therein. Let $\scB^\iota$ be the weighted species of graphs that are $2$-connected, and let $\scC^\omega$ be the weighted species of connected graphs with the $\omega$-weights given as in Equation~\eqref{eq:thweight}.

 Marking a connected graph at a $1$-cycle is equivalent to marking a vertex, and hence we may split $\scC^\circ$ into vertex-marked graphs from the weighted class $\scA^\omega$ and graphs marked with a cycle of length at least two: 
\begin{align}
\label{eq:begin}
(\scC^\circ)^\omega \simeq \scA^\omega + (\scC^\circledast)^\omega.
\end{align}

Let $C$ be connected graph that is marked at a cycle $\tau$ with at least two atoms. Then there exists an automorphism $\sigma$ of $C$ that has $\tau$ as one of its disjoint cycles. The automorphism $\sigma$ induces a canonical  isomorphism $\bar{\sigma}$ of the properly bicolored block-tree $T(C)$ whose, let's say, white vertices correspond to the blocks, and black vertices correspond to cutvertices of $C$. Any vertex $v$ of $C$ corresponds to a unique vertex $\bar{v}$ of $T(C)$, because either $v$ is a cutvertex and hence corresponds to a black vertex of $T(C)$, or $v$ is not a cutvertex, and hence is contained in a unique block of $C$ and hence corresponds to a white vertex of $T(C)$. Since $\tau$ is a cycle of $\sigma$, it follows that the vertices of $T(C)$ that correspond to the atoms of $\tau$ form a cycle $\bar{\tau}$ of the tree-automorphism $\bar{\sigma}$. The cycle $\bar{\tau}$ need not have the same length as the cycle $\tau$, as non-cutvertices of $\tau$ that lie in the same block get contracted to a single atom of $\bar{\tau}$. %has the same length as the cycle $\tau$, except for the case where all vertices of $\tau$ lie in a single block and hence $\bar{\tau}$ is a fixpoint of $\bar{\sigma}$.

\begin{figure}[t]
	\centering
	\begin{minipage}{1.0\textwidth}
		\centering
		\includegraphics[width=0.9\textwidth]{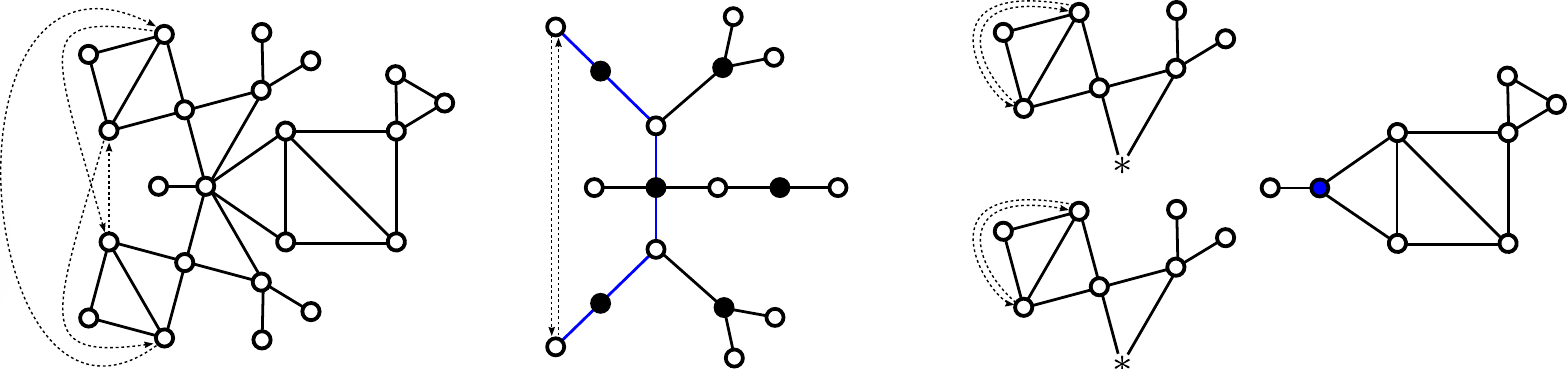}
		\caption{Decomposition of a cycle pointed graph with a cutvertex as cycle center.}
		\label{fi:cycle1}
	\end{minipage}
\end{figure}

For each atom $v$ of $\tau$ we may consider the unique path $P_v$ in the tree $T(C)$ that joins $\bar{v}$ the vertex $\tau(\bar{v})$ corresponding to the consecutive atom in the cycle. As $\bar{\sigma}$ permutes these path, they all have the same lengths. Note that either all or none of the vertices of $\tau$ are cutvertices, as the graph automorphism $\sigma$ permutes only cutvertices with cutvertices and non-cutvertices with non-cutvertices. Hence all vertices of $\bar{\tau}$ share the same colour. In a properly bicolored graph the distance between two vertices of the same colour is always an even number, hence each of the paths $P_v$ has an even number of edges and hence a unique center vertex. A general result given in \cite[Claim 22]{MR2810913} states that all connecting paths in a cycle pointed tree must share the same center, so we may consider {\em the} center vertex $u$ of the connecting paths in $T(C)$. Hence the species $\scC^\circledast$ may be split into two summands,
\begin{align}
\label{eq:split}
(\scC^\circledast)^\omega \simeq  (\scC^\circledast_v)^\omega + (\scC^\circledast_b)^\omega
\end{align}
corresponding to the subspecies where the center of the marked cycles is required to correspond to a cutvertex or to a block, respectively. Clearly the center vertex is a fixpoint of $\bar{\sigma}$, and this fact allows us to give explicit decompositions for both.

Let us first consider the case where the center $u$ corresponds to a cutvertex $v_\tau$. Each branch $A$ of the rooted tree $(T(C), u)$ corresponds to graph $G(A)$ with a distinguished vertex that corresponds to the vertex $v_\tau$ and is not a cut-vertex of $G(A)$. In order to keep the label sets disjoint, we label this vertex by a $*$-place-holder instead of $v_\tau$. Hence any branch is simply a derived block from $\scB'$ where  each non-$*$-vertex gets identified with the root of a connected rooted graph. In other words, its a $\scB' \circ \scA$-object. Moreover, the whole graph consists simply of the center vertex $v_\tau$ together with an unordered symmetrically cycle pointed collection of $\scB' \circ \scA$-objects. The $\omega$-weight of $C$ is the product of the $\omega$-weights of the branches, and each automorphism of $C$ having $\tau$ as its cycle leaves the center vertex $v_c$ invariant. Furthermore, the $\omega$-weight of a branch is the product of the $\iota$-weight of the derived block and the $\omega$-weight of the attached rooted connected graphs. Hence
\begin{align}
\label{eq:vcenter1}
(\scC^\circledast_v)^\omega \simeq (\Set^\circledast \circledcirc ((\scB')^\iota \circ \scA^\omega)) \star \scX,
\end{align}
where the factor $\scX$ corresponds to the center vertex. We may write
\begin{align}
\Set^\circledast \simeq \scP \star \Set
\end{align}
with $\scP$ denoting the cycle pointed species consisting only of marked cycles with length at least two, which simplifies \eqref{eq:vcenter1} to
\begin{align}
\label{eq:cv}
(\scC^\circledast_v)^\omega \simeq (\scP \circledcirc ((\scB')^\iota \circ \scA^\omega)) \star \scA^\omega.
\end{align}
This corresponds to the fact that the center together with the branches without any atoms of $\tau$ form a connected rooted graph without any further restrictions, and the remaining branches together with the marked cycle $\tau$ correspond to a $\scP \circledcirc \scB' \circ \scA$ object. Furthermore, this object may be composed out of a single cycle pointed $(\scB' \circ \scA)^\circ$ object by constructing $\tau$ according to the cycle composition construction, see Figure~\ref{fi:cycle1} for an illustration.

\begin{figure}[t]
	\centering
	\begin{minipage}{1.0\textwidth}
		\centering
		\includegraphics[width=0.9\textwidth]{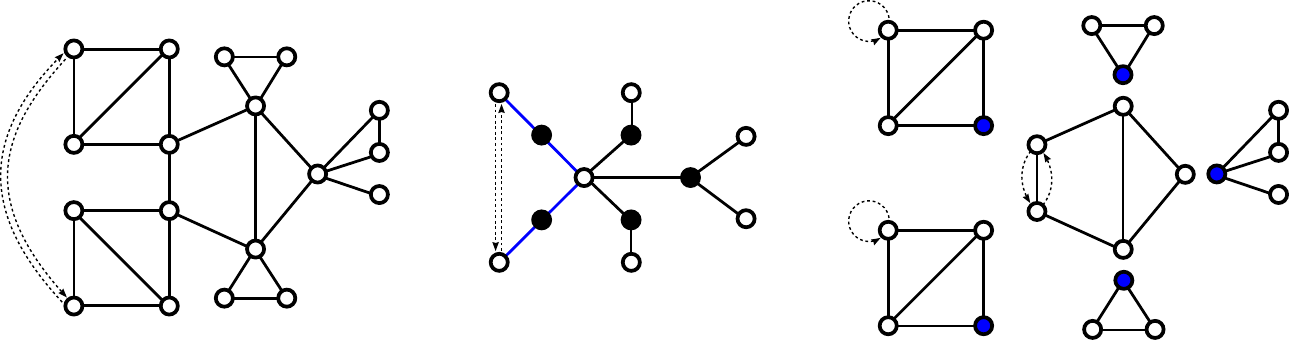}
		\caption{Decomposition of a cycle pointed graph with a block as cycle center.}
		\label{fi:cycle2}
	\end{minipage}
\end{figure}

Finally, consider the case where the center $u$ corresponds to a block $B$ instead of  a cutvertex. There is a natural marked cycle $\tau_B$ on the block $B$. It is given by the cycle $\tau$ if $\tau$ lies entirely in $B$. Otherwise, it is given by the cutvertices of $C$ that are contained in $B$ and belong to those branches in $T(C)$, that are adjacent to the center $u$ and contain atoms of the induced cycle $\bar{\tau}$. This is because the induced automorphism $\bar{\tau}$ of $T(C)$ permutes the branches containing atoms of $\bar{\tau}$ cyclically. The graph automorphism $\sigma$ maps the vertex set of $B$ to itself, and hence induces an automorphism $\sigma|_B$ of the block $B$. The cycle $\tau_B$ is one of the disjoint cycles of $\sigma_B$. Hence  $(B, \tau_B)$ is a cycle pointed block. The graph $C$ may be decomposed into the cycle pointed block $(B, \tau_B)$, where each vertex $v$ of $B$ is identified with the root of a connected rooted graph $C_v$. %The constraints on the attached connected graphs correspond precisely to those of the cycle pointed decomposition, that is, 
The marked cycle $\tau_B$ and the rooted graphs corresponding to it are composed out of a single cycle pointed rooted connected graph $C_\tau$ according to the cycle composition construction (see Figure~\ref{fi:cycle2} below for an illustration). % and any two graphs $C_v$ and $C_w$, that correspond to vertices of $B$ outside of the cycle $\tau_B$, are only allowed to be non-isomorphic, if their roots $v$ and $w$ belong to different disjoint cycles of at at least one automorphism of $B$ that has $\tau_B$ as one its disjoint cycles. 
Furthermore, the $\omega$-weight of the graph $C$ is given by the product of the $\iota$-weight of $B$ and the $\omega$-weights of the attached graphs $(C_v)_{v \in B}$. Summing up, we obtain the decomposition
\begin{align}
\label{eq:block}
(\scC^\circledast_b)^\omega \simeq (\scB^\circledast)^\iota \circledcirc \scA^\omega.
\end{align}

\section{Forming roots}
This section provides a proof for Theorem~\ref{te:approx}, showing that unlabelled unrooted graphs may be approximated in a geometric sense by vertex rooted pendants. Throughout we assume that Assumption~\eqref{eq:H} is satisfied.

\subsection{The case of a vertex cycle-center}

Equation~\eqref{eq:cv} states that an unlabelled symmetrically cycle-pointed weighted connected graph may be decomposed uniquely in a weight-preserving manner into an unlabelled rooted graph and an unlabelled graph from the species $\scP \circledcirc ((\scB')^\iota \circ \scA^\omega)$. We are going to verify that the sum of the weights of $n$-sized graphs from this species is exponentially smaller than the sum of the $\omega$-weights of $n$-sized unlabelled rooted graphs. This may then be used to show that large random unlabelled graphs from $(\scC_v^\circledast)^\omega$ consist of a large rooted graph together with a stochastically bounded rest attached to its root.

Recall that $d$ denotes the span of the support of the generating series $\tilde{\scC}^\omega(z)$. That is, $d \ge 1$ is minimal with the property that the exponents with non-zero coefficients belong to the lattice $1 + d \ndZ$. By a standard result due to  Bell, Burris and Yeats \cite[Thm. 28]{MR2240769}
we know that the tree-like assumption~\eqref{eq:H} implies that there is a constant $c_{\scA}>0$ such that
\begin{align}
\label{eq:asymptotics}
[z^n] \tilde{\scA}^\omega(z) \sim c_{\scA} n^{-3/2} \rho_\scA^{-n}
\end{align}
as $n \equiv 1 \mod d$ becomes large. 

\begin{lemma}
	\label{le:vc}
	The ordinary generating series
	\[
		\widetilde{\scP \circledcirc ((\scB')^\iota \circ \scA^\omega)}(z) = \sum_{i=2}^\infty \bar{Z}_{((\cB')^\circ)^\iota}(\tilde{\scA}^{\omega^i}(z^i), (\tilde{\scA}^\circ)^{\omega^i}(z^i); \tilde{\scA}^{\omega^{2i}}(z^{2i}), (\tilde{\scA}^\circ)^\omega(z^{2i}); \ldots )
	\]
	has radius of convergence strictly larger than $\rho_{\scA}$.
\end{lemma}
\begin{proof}
	For any $k \ge 1$ and any rooted symmetry $(B, \sigma, \tau, v) \in \RSym(((\scB')^\circ)^\iota)[k]$ it holds that
	\begin{align*}	
	[z^b] {(\tilde{\scA}^\circ)^{\omega^{|\tau|}}}(z^{|\tau|}) ({\tilde{\scA}^{\omega}}(z))^{\sigma_1} \cdots  ({\tilde{\scA}^{\omega^{|\tau|}}}(z^{|\tau|}))^{\sigma_{|\tau|}-1}  \cdots ({\tilde{\scA}^{\omega^k}} (z^k))^{\sigma_k} 
\le b [z^b] ({\tilde{\scA}^{\omega}}(z))^{\sigma_1}   \cdots ({\tilde{\scA}^{\omega^k}} (z^k))^{\sigma_k}.
	\end{align*}
Here we have applied the fact that each unlabelled $\scA$-structure of size $\ell\ge1$ has precisely $\ell$ cycle pointings. To any symmetry of size $k$ correspond precisely $k$ rooted symmetries. It follows that for any $b \ge 1$
\begin{multline*}
	[z^b] \sum_{k \ge 1} \frac{1}{k!} \sum_{ (B, \sigma, \tau, v) \in \RSym(((\scB')^\circ)^\iota)[k]} \iota(B){(\tilde{\scA}^\circ)^{\omega^{|\tau|}}}(z^{|\tau|}) ({\tilde{\scA}^{\omega}}(z))^{\sigma_1} \cdots  ({\tilde{\scA}^{\omega^{|\tau|}}}(z^{|\tau|}))^{\sigma_{|\tau|}-1}  \cdots ({\tilde{\scA}^{\omega^k}} (z^k))^{\sigma_k} \\
	\le [z^b] \sum_{k \ge 1} \frac{bk}{k!} \sum_{ (B, \sigma) \in \Sym((\scB')^\iota)[k]} ({\tilde{\scA}^{\omega}}(z))^{\sigma_1}   \cdots ({\tilde{\scA}^{\omega^k}} (z^k))^{\sigma_k}.
\end{multline*}
Only summands with  $[z^b]({\tilde{\scA}^{\omega}}(z))^{\sigma_1}   \cdots{\tilde{\scA}^{\omega^k}} (z^k)^{\sigma_k} \ne0$ contribute. As $\sigma_1 + 2 \sigma_2 + \ldots + k\sigma_k = k$, this means that we only need to consider summands where $k \le b$. Thus
\begin{align}
[z^b]\bar{Z}_{((\cB')^\circ)^\iota}(\tilde{\scA}^\omega(z), (\tilde{\scA}^\circ)^\omega(z); \tilde{\scA}^\omega(z^{2}), (\tilde{\scA}^\circ)^\omega(z^{2}); \ldots ) \le  b^2 [z^b]Z_{(\cB')^\iota}(\tilde{\scA}^\omega(z), \tilde{\scA}^\omega(z^{2}), \ldots).
\end{align}
It follows that for any $\epsilon>0$
\begin{align}
\label{eq:bound}
\widetilde{\scP \circledcirc ((\scB')^\iota \circ (\scA)^\omega)}\left(\rho_\scA + \frac{\epsilon}{2}\right) &\le \sum_{i \ge 2} \sum_{b \ge 1} b^2 \left(\rho_\scA + \frac{\epsilon}{2}\right)^{bi} [z^{bi}] Z_{(\cB')^\iota}(\tilde{\scA}^\omega(z^i), \tilde{\scA}^\omega(z^{2i}), \ldots) \nonumber \\
&=  \sum_{i \ge 2} \frac{1}{i} \sum_{b \ge 1} b^2i \left(\frac{\rho_\scA + \frac{\epsilon}{2}}{\rho_\scA + \epsilon}\right)^{bi} (\rho_\scA + \epsilon)^{bi} [z^{bi}] Z_{(\cB')^\iota}(\tilde{\scA}^\omega(z^i), \tilde{\scA}^\omega(z^{2i}), \ldots).
\end{align}
Clearly it holds that $bi \left(\frac{\rho_\scA + \frac{\epsilon}{2}}{\rho_\scA + \epsilon}\right)^{bi} < 1$ for all but finitely many pairs $(i,b)$. It follows by the assumption~\eqref{eq:H} that the upper bound in \eqref{eq:bound} is finite for $\epsilon$ small enough.
\end{proof}

The asymptotic expansion~\eqref{eq:asymptotics} and Lemma~\ref{le:vc} allow us to apply a standard result~\cite[Thm. VI.12]{MR2483235} on the coefficients of products of power series, yielding
\begin{align}
\label{eq:smallv}
[z^n] (\tilde{\scC}_v^\circledast)^\omega(z) \sim \widetilde{\scP \circledcirc ((\scB')^\iota \circ \scA^\omega)}(\rho_\scA)[z^n]\scA^\omega(z)
\end{align}
as $n$ becomes large. We may apply this to show that large random $(\scC_v^\circledast)^\omega$-objects look like large $\scA^\omega$-objects with a stochastically bounded rest attached to the root.

\begin{lemma}
	\label{le:app1}
The  $\scP \circledcirc ((\scB')^\iota \circ \scA^\omega)$-object corresponding to a random unlabelled $n$-sized $(\scC_v^\circledast)^\omega$-object that is sampled with probability proportional to its weight has stochastically bounded size.
\end{lemma}
\begin{proof}
	By the asymptotic expansions~\eqref{eq:asymptotics} and \eqref{eq:smallv} it follows that the probability for this component to have size $k$ is asymptotically given by 
	\[
		\frac{\left([z^{n-k}]\tilde{\scA}^\omega(z)\right) \left([z^k]\widetilde{\scP \circledcirc ((\scB')^\iota \circ \scA^\omega)}(z)  \right) }{[z^n] \tilde{\scC}_v^\circledast(z)} \to \frac{\rho_\scA^k [z^k]\widetilde{\scP \circledcirc ((\scB')^\iota \circ \scA^\omega)}(z)}{\widetilde{\scP \circledcirc ((\scB')^\iota \circ \scA^\omega)}(\rho_\scA)}.
	\]
	As the limit probabilities sum to $1$, this implies that the component size has a finite weak limit and is hence stochastically bounded. (In fact, this even shows that the component converges weakly to a limit graph following a Boltzmann distribution for unlabelled $\scP \circledcirc ((\scB')^\iota \circ \scA^\omega)$-objects with parameter $\rho_\scA$.) 
\end{proof}

\subsection{The case of a block cycle center}

We start with the following subcriticality observation.

\begin{lemma}
	\label{le:subc}
	The bivariate power series
	\begin{align*}
	f(x,y) := \bar{Z}_{(\scB^\circledast)^\iota}( x, 0; \tilde{\scA}^{\omega^2}(y^2), (\tilde{\scA}^\circ)^{\omega^2}(y^2); \tilde{\scA}^{\omega^3}(y^3), (\tilde{\scA}^\circ)^{\omega^3}(y^3); \ldots  )
	\end{align*}
	satisfies
	$
	f(\tilde{\scA}^\omega(\rho_\scA) + \epsilon, \rho_\scA+ \epsilon) < \infty
	$
	for some $\epsilon>0$.
\end{lemma}
\begin{proof} %It suffices to consider the case $\bar{Z}_{(\scB^\circledast)^\iota} \ne 0$. 
%	The claim follows from the bound
%	\[
%	[x^a y^b] f(x,y) \le b(a+b) [x^a y^b]g(x,y), \qquad a,b \in \ndN_0
%	\]
%	and the tree-like assumption~\eqref{eq:H}. In order to verify this bound, note that 
Any unlabelled $\scA$-structure of size $\ell\ge1$ has precisely $\ell$ cycle pointings. Hence for any $k \ge 1$ and any rooted symmetry $(B, \sigma, \tau, v) \in \RSym(\scB^\circledast)[k]$ it holds that
	\begin{multline*}	
	[x^a y^b] x^{\sigma_1} {(\tilde{\scA}^\circ)^{\omega^{|\tau|}}}(y^{|\tau|}) ({\tilde{\scA}^{\omega^2}}(y^2))^{\sigma_2} \cdots  ({\tilde{\scA}^{\omega^{|\tau|}}}(y^{|\tau|}))^{\sigma_{|\tau|}-1}  \cdots{\tilde{\scA}^{\omega^k}} (y^k)^{\sigma_k} \\
	\le b [x^a y^b] x^{\sigma_1}({\tilde{\scA}^{\omega^2}}(y^2))^{\sigma_2}   \cdots{\tilde{\scA}^{\omega^k}} (y^k)^{\sigma_k}.
	\end{multline*}
	 Any symmetry from $\Sym(\scB)[k]$ corresponds to precisely $k$ rooted symmetries, so any non-trivial symmetry may correspond to at most $k$ rooted symmetry from the symmetrically cycle-pointed species $(\scB^\circledast)^\iota$. Hence
	\begin{align*}
[x^a y^b] f(x, y) &=[x^ay^b] \quad \smashoperator{\sum_{ \substack{k \ge 2 \\ (B, \sigma, \tau, v) \in \RSym(\scB^\circledast)[k]}}} \quad \frac{\iota(B)}{k!}x^{\sigma_1} {(\tilde{\scA}^\circ)^{\omega^{|\tau|}}}(y^{|\tau|}) ({\tilde{\scA}^{\omega^2}}(y^2))^{\sigma_2} \cdots  ({\tilde{\scA}^{\omega^{|\tau|}}}(y^{|\tau|}))^{\sigma_{|\tau|}-1}  \cdots{\tilde{\scA}^{\omega^k}} (y^k)^{\sigma_k} \\
	&\le [x^ay^b] \sum_{k \ge 2} \frac{bk}{k!} \sum_{(B, \sigma) \in \Sym(\scB)[k] } \iota(B)x^{\sigma_1}({\tilde{\scA}^{\omega^2}}(y^2))^{\sigma_2}   \cdots{\tilde{\scA}^{\omega^k}} (y^k)^{\sigma_k}.
\end{align*}
	We may neglect any summands where $\sigma_1 \ne a$ or $[y^b]({\tilde{\scA}^{\omega^2}}(y^2))^{\sigma_2}   \cdots{\tilde{\scA}^{\omega^k}} (y^k)^{\sigma_k} =0$. Since it holds that $\sigma_1 + 2 \sigma_2 + \ldots + k\sigma_k = k$, this means that we only need to consider summands where $k \le a+b$. Thus
	\begin{align}
		\label{eq:a}
		[x^ay^b]f(x,y) \le b(a+b) [x^ay^b] Z_{\scB^\iota}(x, \tilde{\scA}^{\omega^2}(y^2), \tilde{\scA}^{\omega^3}(y^3), \ldots). 
	\end{align}
	It follows from the identity $Z_{(\scB')^\iota}(s_1, s_2,\ldots) = \frac{\partial}{\partial s_1} Z_{\scB^\iota}(s_1, s_2,\ldots)$  that for $a>0$ 
	\begin{align}
		\label{eq:b1}
		[x^ay^b] f(x,y) \le \frac{b(a+b)}{a} [x^{a-1}y^{b}]Z_{(\scB')^\iota}(x, \tilde{\scA}^{\omega^2}(y^2), \tilde{\scA}^{\omega^3}(y^3), \ldots).
	\end{align}
	In order to treat the case $a=0$, we observe that the series $Z_{\scB^\iota}(0, \tilde{\scA}^{\omega^2}(y^2), \tilde{\scA}^{\omega^3}(y^3), \ldots)$ is the sum of weight-monomials of all fixed-point-free symmetries of block-rooted connected graphs. We may convince  ourselves of this fact as follows. Block-rooted connected graphs consist of a block with rooted graphs attached to it, so they correspond to the composition species $\scB^\iota \circ \scA^\omega$. By the composition formula the cycle index sum of this species is given by
	\[
	Z_{\scB^\iota}(Z_{\scA^\omega}(s_1,s_2, \ldots), Z_{\scA^{\omega^2}}(s_2,s_4,\ldots),Z_{\scA^{\omega^3}}(s_3,s_6,\ldots) \ldots).
	\] 
	If we want to sum only the weight-monomials of fixed-point-free symmetries, we have to make the substitution $s_1=0$. But $Z_{\scA^\omega}(0,s_2, \ldots)=0$ as any automorphism of a rooted graph from $\scA^\omega$ is required to fix the root. So 
		$
	Z_{\scB^\iota}(0, Z_{\scA^{\omega^2}}(s_2,s_4,\ldots),Z_{\scA^{\omega^3}}(s_3,s_6,\ldots) \ldots).
	$
	is the sum of weight-monomials of fixed-point-free symmetries of block-rooted graphs. If we want to index according to the number of vertices we have to make the substitution $s_i=y^i$ for all $i \ge 2$, yielding $Z_{\scB}(0, \tilde{\scA}^{\omega^2}(y^2), \tilde{\scA}^{\omega^3}(y^3), \ldots)$. 
	
	Now, any connected graph with $b$ vertices has at most $b$ blocks. Hence this series counts each fixed-point-free symmetry of a connected graph  with $b$ vertices (without a block-root) at most $b$ times, yielding
	\[
	[y^b]Z_{\scB^\iota}(0, \tilde{\scA}^{\omega^2}(y^2), \tilde{\scA}^{\omega^3}(y^3), \ldots) \le b [y^b] Z_{\scC^\omega}(0, y^2, y^3, \ldots).
	\]
	By~\eqref{eq:a} we may deduce
	\begin{align}
		\label{eq:b2}
		[y^b] f(0,y) \le b^3 Z_{\scC^\omega}(0, y^2, y^3, \ldots).
	\end{align}
	It follows from the bounds~\eqref{eq:b1}, \eqref{eq:b2} and the  tree-like requirement \eqref{eq:H} that $f(  \tilde{\scA}^\omega(\rho_\scA) + \epsilon, \rho_\scA + \epsilon) < \infty$ for some $\epsilon>0$.
\end{proof}

Note that it may happen that $\bar{Z}_{((\scB')^\circledast)^\iota} =0$. This is the case if we only assign positive $\iota$-weights to graphs with the property, that any automorphism with a fixed-point must be the trivial automorphism. There are even graphs like the Frucht graph who only admit the trivial automorphism, hence we have to be mindful of this possibility.

\begin{lemma}
	\label{le:analytic}
	If $\bar{Z}_{((\scB')^\circledast)^\iota} =0$, then the generating series $(\tilde{\scC}_b^\circledast)^\omega(z)$ is analytic at $\rho_\scA$.	
\end{lemma}
\begin{proof}
	The assumption $\bar{Z}_{((\scB')^\circledast)^\iota} =0$ implies that $\bar{Z}_{(\scB^\circledast)^\iota}(x_1,y_1; x_2,y_2; \ldots) = \bar{Z}_{(\scB^\circledast)^\iota}(0,0; x_2,y_2; \ldots)$. Hence $(\tilde{\scC}_b^\circledast)^\omega(z) = f(0,z)$. By Inequality~\eqref{eq:b2} and the tree-like requirement \eqref{eq:H} we know that $f(0,z)$ has radius of convergence strictly larger than $\rho_\scA$, and consequently so does $(\tilde{\scC}_b^\circledast)^\omega(z)$.
\end{proof}

Let us assume for the remaining part of this subsection that $\bar{Z}_{((\scB')^\circledast)^\iota} \ne 0$. In~\cite[Thm. 3.1, Lem. 3.2]{2016arXiv161001401S} general results for the behaviour of component sizes and partitions functions of unlabelled composite structures were given. Lemma~\ref{le:subc} gives an analogous subcriticality condition to this setting, but for the cycle-pointed composition $(\scB^\circledast)^\iota \circledcirc \scA^\omega$ rather than a regular composition. However, the arguments used in \cite{2016arXiv161001401S} may be modified  to encompass the present setting. In the following we describe these modifications.

Decomposition~\eqref{eq:block} allows us to apply the substitution rule for Boltzmann samplers given in~\cite[Fig. 13]{MR2810913} in order to devise a sampling procedure for graphs from the class $(\scC_b^\circledast)^\omega$. (Be mindful that the arXiv version and journal version of \cite{MR2810913} have different numbering of theorems and figures. We refer to the version that got published by SIAM J. Comput. Furthermore, the results of \cite{MR2810913} are stated in a setting of species without weights, but the generalization to weighted species is straight-forward.)  This yields the following procedure which samples a random unlabelled $(\scC_b^\circledast)^\omega$-object $\mC$ with distribution given by \begin{align}
\label{eq:cc}\Pr{\mC = C} = \omega(C) \rho_\scA^{|C|} / \tilde{\scC}_b^\circledast(\rho_\scA).\end{align}
\begin{enumerate}
	\item Draw a rooted symmetry $(B, \sigma, \tau, v) \in \bigcup_{k \ge 0} \RSym(\scB^\circledast)[k]$ with probability proportional to the weight $\frac{\iota(B)}{|B|!} {(\tilde{\scA}^\circ)^{\omega^{|\tau|}}}(\rho_\scA^{|\tau|})(\tilde{\scA}^\omega(\rho_\scA))^{\sigma_1} ({\tilde{\scA}^{\omega^2}}(\rho_\scA^2))^{\sigma_2} \cdots  ({\tilde{\scA}^{\omega^{|\tau|}}}(\rho_\scA^{|\tau|}))^{\sigma_{|\tau|}-1}  \cdots{\tilde{\scA}^{\omega^{|B|}}} (\rho_\scA^{|B|})^{\sigma_{|B|}}$.
	\item For each unmarked cycle $c \ne \tau$ of $\sigma$ draw an unlabelled $\scA^\omega$-object $A_c$ with probability proportional to the weight $\omega(A_c)^{|c|}\rho_\scA^{|c||A_c|}$. Draw a cycle-pointed graph $(A_\tau, c_\tau)$ from the unlabelled $(\scA^\circ)^\omega$-objects with probability proportional to the weight $\omega(A_\tau)^{|\tau|}\rho_\scA^{|\tau| |A_\tau|}$.
	\item Construct the final graph $\mC$ by identifying for each cycle $c$ of $\sigma$ and each atom $u \in c$ (which is a vertex of $B$) the vertex $u$ with the root of a copy of $A_c$. The marked cycle of $\mC$ has length $|\tau||c_\tau|$ and is constructed in a certain way out of the atoms of the $|\tau|$ copies of the cycle $c_\tau$. (The precise way of composing this cycle is irrelevant for our following arguments. Hence we refer the reader to~\cite[Fig. 13]{MR2810913} for details.)
\end{enumerate}

We may split the third step into two steps 3' and 3'', where in step 3' we treat only cycles $c$ of $\sigma$ of length at least two, and in step 3'' we attach only the graphs $A_c$ for $c$ a fixed-point of $\sigma$. This way, we end up with a graph $\mH$ in step 3' having a number $F$ of marked vertices, each of which gets identified in step 3'' with the root of an independent copy of a random unlabelled $\scA$-object $\mA$ with distribution
\[
\Pr{\mA = A} = \omega(A)\rho_\scA^{|A|} / \tilde{\scA}^\omega(\rho_\scA).
\]
The joint probability generating series for $F$ and $H := |\mH| -F$  is given by
\begin{align}
	\label{eq:joint}
	\Ex{x^F y^H} = f( x\tilde{\scA}^\omega(\rho_\scA), y\rho_\scA ) / (\tilde{\scC}_b^\circledast)^\omega(\rho_\scA).
\end{align}
(Recall that the bivariate power series $f$ was defined in Lemma~\ref{le:subc}.) Lemma~\ref{le:subc} ensures that the vector $(F,H)$ has finite exponential moments. Let $(\mA_i)_{i \ge 1}$ denote independent copies of $\mA$. Then
\begin{align}
	\label{eq:comp}
	|\mC| \eqdist H + \sum_{i=1}^F |\mA_i|.
\end{align}
We are now in the same situation as in \cite[Eq. (4.2)]{2016arXiv161001401S}, yielding by analogous arguments as for \cite[Eq. (4.4)]{2016arXiv161001401S} that
\begin{align}
	\label{eq:pc}
	\Pr{|\mC|=n} \sim \Ex{F} \Pr{|\mA| = n}
\end{align}
as $n \equiv 1 \mod d$ tends to infinity. Using the identity $\bar{Z}_{((\scB')^\circledast)^\iota}(x_1, y_1; x_2, y_2; \ldots) = \frac{\partial}{\partial x_1} \bar{Z}_{(\scB^\circledast)^\iota}(x_1, y_1; x_2, y_2; \ldots)$  this may be expressed in terms of coefficients of generating series by
\begin{align}
	\label{eq:T3}
	\frac{[z^n] (\tilde{\scC}_b^\circledast)^\omega(z)}{[z^n]\tilde{\scA}^\omega(z)} \to \widetilde{((\scB')^\circledast)^\iota \circledcirc {\scA}^\omega}(\rho_\scA) .
\end{align}
This also holds in the case $\bar{Z}_{((\scB')^\circledast)^\iota} \ne 0$, since then $\widetilde{((\scB')^\circledast)^\iota \circledcirc {\scA}^\omega}(\rho_\scA)=0$. 
As we are in the setting~\eqref{eq:comp}, we may also argue entirely analogously as in the proof of \cite[Thm. 3.1]{2016arXiv161001401S} to obtain the following result.

\begin{lemma}
	\label{le:app2}
	If $\bar{Z}_{((\scB')^\circledast)^\iota} \ne 0$, then
	$
		( \max(|\mA_1|, \ldots, |\mA_F|) \mid |\mC|= n) = n + O_p(1).
	$
\end{lemma}
That is, if we draw a graph $\mC_n^b$ with probability proportional to its weight among all unlabelled $n$-vertex graphs from the class $(\tilde{\scC}_b^\circledast)^\omega$, then $\mC_n^b \eqdist (\mC \mid |\mC|=n)$ consists of large rooted graph $\mA_n^b$ (the component $\mA_i$ with maximal size) with a small graph of stochastically bounded size attached to its root. Furthermore, for any $k \ge 1$ it holds that the conditioned rooted component $(\mA_n^b \mid |\mA_n^b| = k)$ gets sampled with probability proportional to its $\omega$-weight among all unlabelled rooted graphs with size $k$.

\subsection{Enumerative properties and a justification of Corollary~\ref{co:components}}

It follows from Decompositions~\eqref{eq:begin} and~\eqref{eq:split}, and Equation~\eqref{eq:smallv}, Lemma~\ref{le:analytic} and Equation~\eqref{eq:T3} that
\begin{align}
\label{eq:cyc}
 [z^n] \tilde{\scC}^\omega(z) &= n^{-1} [z^n]\left(\tilde{\scA}^\omega(z) + (\tilde{\scC}^\circledast_v)^\omega(z) + (\tilde{\scC}^\circledast_b)^\omega(z)\right) \nonumber \\
 	&\sim \left(1 + \widetilde{\scP \circledcirc ((\scB')^\iota \circ \scA^\omega)}(\rho_\scA) +  \widetilde{((\scB')^\circledast)^\iota \circledcirc {\scA}^\omega}(\rho_\scA) \right) n^{-1} [z^n]\tilde{\scA}^\omega(z) \nonumber \\
 	&\sim c_{\scA} \left(1 + \widetilde{\scP \circledcirc ((\scB')^\iota \circ \scA^\omega)}(\rho_\scA) +  \widetilde{((\scB')^\circledast)^\iota \circledcirc {\scA}^\omega}(\rho_\scA) \right) n^{-5/2} \rho_\scA^{-n}
\end{align}
as $n \equiv 1 \mod d$ becomes large.
The result \cite[Thm. 28]{MR2240769} by  Bell, Burris and Yeats yields an explicit expression for the constant $c_\scA$ in Equation~\eqref{eq:asymptotics}, namely
\begin{align}
c_\scA = d \sqrt{\frac{\rho_\scA E_z(\rho_\scA, \tilde{\scA}^\omega(\rho_\scA))}{2\pi E_{uu}(\rho_\scA, \tilde{\scA}^\omega(\rho_\scA))}}
\end{align}
with $E_z$ and $E_{uu}$ denoting partial derivatives of the bivariate power series
\begin{align}
\label{eq:EE}
E(z, u) := z \exp\left( Z_{ (\cB')^\iota}(u, \tilde{\scA}^{\omega^2}(z^2),\tilde{\scA}^{\omega^3}(z^3), \ldots)  + \sum_{i \ge 2} \frac{1}{i}Z_{ (\cB')^{\iota^i}}(\tilde{\scA}^{\omega^i}(z^i), \tilde{\scA}^{\omega^{2i}}(z^{2i}),\ldots)\right).
\end{align}

For $d=1$, it  follows for example by~\cite[Lem. 3.2]{2016arXiv161001401S} that
\begin{align}
	[z^n] \tilde{\scG}^\omega(z) \sim \scG^\omega(\rho_\scA) [z^n] \tilde{\scC}^\omega(z)
\end{align}
with $\scG^\omega(\rho_\scA) = \exp(\sum_{i \ge 1} \tilde{\scC}^\omega(\rho_\scA^i) / i)$. Hence we recover the asymptotic expansion obtained in \cite[Thm. 15]{MR2873207}. By~\cite[Thm. 3.1]{2016arXiv161001401S} it follows that the largest connected component of the random graph $\mG_n^\omega$ has size $n + O_p(1)$, and that the remaining small fragments $\text{frag}(\mG_n^\omega)$ satisfy the limit
\begin{align}
	\text{frag}(\mG_n^\omega) \convdis \mG_0,
\end{align}
with $\mG_0$ defined in Corollary~\ref{co:components}. 

In the general case $d \ge 1$, the modulo of $n$ imposes restrictions on the number of components, as connected components whose number of vertices does not lie in $1 + d\ndN_0$ have weight zero. Thus for $n \equiv a \mod d$, $0 \le a < d$ an $n$-sized unlabelled graph from $\scG^\omega$ with non-zero weight is a multiset of connected unlabelled graphs with the total number of elements belonging to $a + d \ndN_0$. (The converse is ensured to hold when all connected components are sufficiently large, since there are only finitely many unlabelled connected graphs with size in $1 + d \ndN_0$ and weight zero.) We let $\Set_a$  denote the species with a single unlabelled object of size $k$ for each $k \in a + d\ndN_0$. It is easy to generalize~\cite[Thm. 3.1, Lem. 3.2]{2016arXiv161001401S} to obtain
\begin{align}
[z^n] \tilde{\scG}^\omega(z) \sim  Z_{\Set_a'}(\tilde{\scA}^\omega(\rho_\scA), \tilde{\scA}^{\omega^2}(\rho_\scA^2), \ldots) [z^{n+1-a}] \tilde{\scC}^\omega(z) \quad \text{and} \quad \text{frag}(\mG_n^\omega) \convdis \mG_a
\end{align}
as $n$ tends to infinity along the lattice $a + d \ndN_0$, with $\mG_a$ the random graph defined in Corollary~\ref{co:components}. (Compare with~\cite[Thm. 3.4]{Mreplaceme}, where such a generalization was carried out in the labelled setting.) This finalizes the proof of Corollary~\ref{co:components}.

\subsection{A proof of Theorem~\ref{te:approx}}
Suppose that $\bar{Z}_{((\scB')^\circledast)^\iota} \ne0$. In order to prove Theorem~\ref{te:approx} it suffices by Decompositions~\eqref{eq:begin} and~\eqref{eq:split} to show such approximation statements for random unlabelled $n$-vertex graphs sampled with probability proportional to their weight from the classes $\scA^\omega$, $(\scC^\circledast_v)^\omega$  and $ (\scC^\circledast_b)^\omega$. For the class $\scA^\omega$ this is trivial, and for the other two classes this is precisely what we did in Lemma~\ref{le:app1} and Lemma~\ref{le:app2}. Hence Theorem~\ref{te:approx} holds in this case, and we even obtain that $\mC_n^\omega \simeq \mD_n + \mA_{n-d_n}^\omega$. That is, the upper bound for the total variational distance in   Inequality~\eqref{eq:tv} is equal to zero.

In the case $\bar{Z}_{((\scB')^\circledast)^\iota} = 0$ we may argue again that analogous approximations as in Theorem~\ref{te:approx} hold for the classes $\scA^\omega$ and $(\scC^\circledast_v)^\omega$. However, such a statement does not appear to hold any longer for the class~$ (\scC^\circledast_b)^\omega$. This is not a problem, as Lemma~\ref{le:analytic} and Equation~\eqref{eq:cyc} guarantee that a random $n$-vertex cycle-pointed connected graph sampled with probability proportional to its $\omega$-weight belongs only with exponentially small probability to the class~$ (\scC^\circledast_b)^\omega$. Hence Theorem~\ref{te:approx} follows, and Inequality~\eqref{eq:tv} holds with the upper bound of the total variational distance being given by the quotient $\frac{[z^n](\tilde{\scC}^\circledast_b)^\omega(z)}{n [z^n]\tilde{\scC}^\omega(z)} \le C \exp(-cn)$ uniformly in all $n$ for fixed constants $C,c>0$ that do not depend on $n$.

\subsection{Proof strategy for Theorem~\ref{te:scalinglimit}}

The proof for Theorem~\ref{te:scalinglimit} starts with the following main lemma, which extends the Gromov--Hausdorff limit \cite[Thm. 6.14]{2015arXiv150402006S} to convergence in the rooted Gromov--Hausdorff--Prokhorov sense.

\begin{lemma}
	\label{le:lemma}
	Suppose that the tree-like requirement~\eqref{eq:H} is satisfied. Then there is a constant $c_\omega>0$ such that the random rooted graph $\mA_n^\omega$ equipped with the uniform measure $\mu_n^\cA$ on its set of vertices satisfies the weak limit
	\[
		\left(\mA_n^\omega, \frac{c_\omega}{\sqrt{n}} d_{\mA_n^\omega}, \mu_n^\scA \right) \convdis (\CRT, d_{\CRT}, \mu)
	\]
	in the rooted Gromov--Hausdorff--Prokhorov sense. 
\end{lemma}

The justification of this fact requires us to recall additional notation. Hence we postpone a detailed to proof of Lemma~\ref{le:lemma} to Section~\ref{sec:rghp} below. As the graph $\mD_n$ from Theorem~\ref{te:approx} has size $d_n = O_p(1)$ it follows from Lemma~\ref{le:lemma} and the diameter tail-bound 
\begin{align}
	\label{eq:dum}
	\Pr{\Di(\mA_n^\omega) \ge x} \le C \exp(-cx^2/n)
\end{align}
from~\cite[Thm. 6.14]{2015arXiv150402006S} that
\[
\left(\mA_{n-d_n}^\omega, \frac{c_\omega}{\sqrt{n}} d_{\mA_{n-d_n}^\omega}, \mu_{n-d_n}^\scA \right) \convdis (\CRT, d_{\CRT}, \mu).
\]
Let $\mu_n^{\scA+\scD}$ denote the uniform measure on the vertex set of the graph  $\mA_n^\omega + \mD_n$. Using again  $d_n = O_p(1)$ we may observe
\[
	d_{\text{GHP}} \left(	\left(\mA_{n-d_n}^\omega + \mD_n, \frac{c_\omega}{\sqrt{n}} d_{\mA_{n-d_n}^\omega + \mD_n}, \mu_{n}^{\scA + \scD} \right), \left(\mA_{n-d_n}^\omega, \frac{c_\omega}{\sqrt{n}} d_{\mA_{n-d_n}^\omega}, \mu_{n-d_n}^\scA \right) \right)  \convp 0.
\]
This verifies the  limit~\eqref{eq:limt}.

In order to complete the proof of Theorem~\ref{te:scalinglimit}, it remains
to establish a  tail-bound for the diameter.  It suffices to verify a bound as in~\eqref{eq:tailb} uniformly for $\sqrt{n} \le x \le n$. Moreover, we may treat the three individual parts of the decomposition $(\scC^\circ)^\omega \simeq \scA^\omega + (\scC^\circledast_v)^\omega + (\scC^\circledast_b)^\omega$ in \eqref{eq:begin} and \eqref{eq:split} individually. Inequality~\eqref{eq:dum} takes care of the first summand $\scA^\omega$. As for the case of a vertex cycle-center, we may sample an $n$-vertex unlabelled graph from $(\scC^\circledast_v)^\omega$ with probability proportional to its weight by conditioning the following procedure on producing a graph with size $n$:
\begin{enumerate}
	\item Draw a random unlabelled $\scA^\omega$-object $\mA$ with probability $\Pr{\mA = A} = \omega(A) \rho_\scA^{|A|} / \tilde{\scA}^\omega(\rho_\scA)$.
	\item Draw a random unlabelled ${\scP \circledcirc ((\scB')^\iota \circ \scA^\omega)}$-object $\mP$ with $\Pr{\mP =P}$  proportional to $\omega(P) \rho_\scA^{|P|}$.%$\widetilde{\scP \circledcirc ((\scB')^\iota \circ \scA^\omega)}(\rho_\scA)$.
	\item Glue $\mA$ and $\mP$ together at their root vertices to form the graph $\mA + \mP$.
\end{enumerate}
Note that the total size of the graph $\mA + \mP$ is $|\mA| + |\mP| -1$ as we identify the two roots. By \eqref{eq:asymptotics} and \eqref{eq:smallv} we know that $\Pr{|\mA| + |\mP| -1 = n} = \rho_\scA^n [z^n](\tilde{\scC}^\circledast_v)^\omega(z) / (\tilde{\scC}^\circledast_v)^\omega(\rho_\scA) = O(n^{3/2})$. If $\Di(\mA + \mP) \ge x$ then it holds that $\Di(\mA) \ge x/2$ or $|\mP| \ge x/2$. It follows that
\begin{align}
\label{eq:t1}
\Pr{\Di(\mA + \mP) \ge x \mid |\mA| + |\mP| = n-1} & \le \Pr{\Di(\mA) \ge x/2 \mid |\mA| + |\mP| = n-1} + O(n^{3/2}) \Pr{|\mP| \ge x/2}. 
\end{align}
By Lemma~\ref{le:vc} there are constants $C',c'>0$ such that $\Pr{|\mP| \ge y} \le C' \exp(-c'y)$ for all $y$.  It follows that uniformly for all $\sqrt{n} \le x \le n$ we have
\begin{align}
\label{eq:t2}
O(n^{3/2}) \Pr{|\mP| \ge x/2}  \le O(n^{3/2}) \exp(-c' x/2) = O(1) \exp(- c'(1 + o(1))x/2) \le C'' \exp(-c'' x^2 /n)
\end{align}
for some fixed constants $C'', c'' >0$. Using~\eqref{eq:dum} and $x \le n$ we obtain also that
\begin{align}
\label{eq:t3}
	\Pr{\Di(\mA) \ge x/2 \mid |\mA| + |\mP| = n-1} &\le \sum_{k=x/2}^n \Pr{\Di(\mA) \ge x/2 \mid |\mA|=k} \Pr{|\mA|=k \mid  |\mA| + |\mP| = n-1} \nonumber \\ 
	&\le C \sum_{k=x/2}^n \exp(-cx^2/(4k)) \Pr{|\mA|=k \mid  |\mA| + |\mP| = n-1} \nonumber\\
	&\le C \exp(-cx^2/(4n)).
\end{align}
Combining \eqref{eq:t1}, \eqref{eq:t2} and \eqref{eq:t3} it follows that
\[
	\Pr{\Di(\mA + \mP) \ge x \mid |\mA| + |\mP| = n-1} \le C''' \exp(-c''' x^2 /n)
\]
uniformly for $\sqrt{n} \le x \le n$ for some constants $C''', c''' > 0$. 

It remains to verify such a bound for the case of a block-cycle center. By Lemma~\ref{le:analytic} it follows that in case $\bar{Z}_{((\scB')^\circledast)^\iota} =0$ the probability for a random unlabelled $n$-vertex graph from the class $(\scC^\circ)^\omega$  have a block cycle-center is exponentially small, that is, it is bounded by $C_1 \exp(-c_1 n)$ from some constants $C_1, c_1>0$ that do not depend on $n$. As $x \le n$ we have $C_1 \exp(-c_1 n) \le C_1 \exp(-c_1 x^2 / n)$, and hence we are done in this case.

In case $\bar{Z}_{((\scB')^\circledast)^\iota} \ne 0$ the strategy is similar to the case of a vertex cycle center, but the details are more technical. Recall the random graph $\mC$ from Equation~\eqref{eq:cc} and its sampling procedure in subsequent paragraphs that splits it into a part with $H$ vertices and $F$ rooted components $\mA_1, \ldots, \mA_F$  as in \eqref{eq:comp}. If $\Di(\mC) \ge x$ then it holds that $H \ge x/2$ or $\max(\Di(\mA_1), \ldots, \Di(\mA_F)) \ge x/4$. As $\Pr{|\mC|=n} = O(n^{3/2})$ by \eqref{eq:pc} and \eqref{eq:asymptotics}, it follows that
\begin{align}
	\label{eq:s1}
	\Pr{\Di(\mC) \mid |\mC| =n} \le \Prb{\max(\Di(\mA_1), \ldots, \Di(\mA_F)) \ge x/4 \mid |\mC|= n} + O(n^{3/2}) \Pr{ H \ge x/2}.
\end{align}
Lemma~\ref{le:subc} ensures that the vector $(F,H)$ has finite exponential moments. In particular there are constants $C_1', c_1' >0$ such that $\Pr{H \ge y} \le C_1' \exp(-c_1' y)$ uniformly for all $y$. Hence
\begin{align}
	\label{eq:s2}
O(n^{3/2}) \Pr{ H \ge x/2} \le O(n^{3/2}) \exp(-c_1' x/2) = O(1) \exp(- c_1'(1 + o(1))x/2) \le C_1'' \exp(-c_1'' x^2 /n)
\end{align}
for some constants $C_1'', c_1''>0$. As for the other summand in \eqref{eq:s1}, for any $f',k_1, \ldots, k_{f'} \ge 0$ let $\cE$ denote the event that $F= f'$ and $|\mA_i| = k_i$ for all $1 \le i \le f'$. We may argue using \eqref{eq:dum} that
\begin{align}
\label{eq:s3}
\Prb{\max_{1 \le i \le F}\Di(\mA_i) \ge x/4 \mid |\mC|= n}  &\le \sum_{{ 1 \le f', k_1, \ldots, k_{f'} \le n}}\Prb{\max_{1 \le i \le F}\Di(\mA_i) \ge x/4 \mid \cE} \Pr{\cE \mid |\mC| = n}  \nonumber \\
& \le \sum_{{ 1 \le f', k_1, \ldots, k_{f'} \le n}} \sum_{1 \le i \le f'} C \exp(- c x^2/(16 k_i)) \Pr{\cE \mid |\mC| = n}\nonumber \\
& \le  C \exp(- c x^2/(16 n)) \sum_{{ 1 \le f', k_1, \ldots, k_{f'} \le n}} f' \Pr{\cE \mid |\mC| = n}\nonumber \\
& \le  C \exp(- c x^2/(16 n)) \Ex{F \mid |\mC| = n}.
\end{align}
It follows from the expression~\eqref{eq:joint} of the joint probability generating function of $F$ and $H$ that
\[
\Ex{F \mid |\mC| = n} = \frac{[z^n] \frac{\partial f}{\partial x}(\tilde{\scA}^\omega(z \rho_\scA), z \rho_\scA )}{ [z^n] f(\tilde{\scA}^\omega(z \rho_\scA), z \rho_\scA )}.
\]
We know by~\eqref{eq:T3} that $[z^n] f(\tilde{\scA}^\omega(z \rho_\scA), z \rho_\scA )$ is asymptotically equivalent to $[z^n] \tilde{\scA}^\omega(z)$ up to a constant factor. By the same arguments (just with $\frac{\partial f}{\partial x}$ instead of $f$) the same holds for $[z^n] \frac{\partial f}{\partial x}(\tilde{\scA}^\omega(z \rho_\scA), z \rho_\scA )$. (This could also be verified using the usual singularity analysis methods from \cite[Thm. VI.5]{MR2483235}.) Consequently, the conditional expectation $\Ex{F \mid |\mC| =n}$ remains bounded as $n$ tends to infinity. It follows from this fact and Inequalities~\eqref{eq:s1}, \eqref{eq:s2}, and \eqref{eq:s3} that $\Pr{\Di(\mC) \mid |\mC| =n} \le C_1'' \exp(-c_1'' x^2/n)$ holds uniformly in $\sqrt{n} \le x \le n$ for some constants $C_1'', c_1''>0$. This completes the proof of Theorem~\ref{te:scalinglimit}.

\subsection{A proof of Theorem~\ref{te:bslimit}}

It was shown in \cite[Thm. 6.13]{2015arXiv150402006S} that there is a random rooted graph $\hat{\mC}$ such that for any sequence $k_n = o(\sqrt{n})$ the $k_n$ neighbourhood $U_{k_n}(\mA_n^\omega, u_n)$ of a uniformly selected vertex $u_n \in \mA_n^\omega$ satisfies
\[
	d_{\mathrm{TV}}( U_{k_n}(\mA_n^\omega, u_n), U_{k_n}(\hat{\mC})) \to 0.
\]
This was obtained from a more general result~\cite[Thm. 6.8]{2015arXiv150402006S}, that also yields that the root of $\mA_n^\omega$ is with high probability \emph{not} contained in $U_{k_n}(\mA_n^\omega, u_n)$. By Theorem~\ref{te:approx} we know that a uniformly selected vertex $x_n$ of $\mD_n + \mA^\omega_{n - d_n}$ lies with high probability in $\mA^\omega_{n - d_n}$, since $\mD_n$ accounts for stochastically bounded subset of the $n$ vertices. Conditioned on this event, the vertex $x_n$ is uniformly distributed among the vertices of $\mA^\omega_{n - d_n}$. Since $k_n / \sqrt{n - d_n} = o_p(1)$, it follows that with high probability the $k_n$ neighbourhood of $x_n$ does not contain the root of $\mA^\omega_{n - d_n}$. That is, $U_{k_n}(\mD_n + \mA^\omega_{n - d_n}, x_n) = U_{k_n}( \mA^\omega_{n - d_n}, x_n)$ holds with probability tending to $1$ as $n$ becomes large. It follows from \eqref{eq:tv} that the uniformly selected vertex $v_n \in \mC_n^\omega$ satisfies
\[
d_{\mathrm{TV}}( U_{k_n}(\mC_n^\omega, u_n), U_{k_n}(\hat{\mC})) \to 0.
\]
This proves Theorem~\ref{te:bslimit}.

\section{Convergence in the rooted GHP sense}
\label{sec:rghp}

In the proof of Theorem~\ref{te:scalinglimit} we postponed a justification of Lemma~\ref{le:lemma} to this section, as it requires us to recall some notation and results. Throughout we assume that the tree-like requirement~\eqref{eq:H} is satisfied.

It was shown in~\cite[Lem. 6.1, 6.2]{2015arXiv150402006S} in a more general context that the random graph $\mA$ with distribution $\Pr{\mA = A} = \omega(A) \rho_\scA^{|A|} / \tilde{\scA}^\omega(\rho_\scA)$ may be sampled according to a certain process. We present a description of this process using a slightly simplified notation, so we have to recall only the details that we are actually going to use. It involves a certain random connected graph $\mG$ that has one marked root that by convention, does not contribute to its size, and the remaining vertices are partitioned into a set $f(\mG)$ of "fixed-points" and  "non-fixed-points" $F(\mG)$. (See \cite{2015arXiv150402006S} for context on why it makes sense to use this terminology.) The process goes by starting with a root-vertex $o$ and identifying it with the root of an independent copy $G(o)$ of $\mG$. Then for each fixed-point $v$ we take a fresh independent copy $G(v)$ of $\mG$, identify $v$ with the root of $G(v)$, and mark $v$ as visited. The process continuous in this way until it dies out, that is, when there are no unvisited fixed-points left. This happens almost surely and the resulting graph is distributed like $\mA$. In the following, we may assume that $\mA$ is actually sampled in this way. By convention, we will also refer to the root of $\mA$ as a fixed-point.

The graph $\mG$ is defined in \cite[Sec. 6]{{2015arXiv150402006S}} in such a way such that the bivariate generating function of the sizes $\xi := |f(\mG)|$ and $\zeta := |F(\mG)|$ is given by
\begin{align}
	\label{eq:def}
	\Ex{z^\xi w^\zeta} = \exp\left( z \tilde{\scA}^\omega(\rho_\scA) + \sum_{i=2}^\infty \tilde{\scA}^{\omega^i}(\rho_\scA^i w^i)\right) \rho_\scA / \tilde{\scA}^\omega(\rho_\scA).
\end{align}
By \cite[Lem. 6.3]{2015arXiv150402006S} the vector $(\xi, \zeta)$ has finite exponential moments and it holds that $\Ex{\xi}=1$.  The sampling procedure for $\mA$ resembles a branching process. Indeed, we may form the tree $\cT^f$ consisting of the fixed-points of $\mA$ such the offspring of a fixed-point $v$ is given by the fixed-point of its associated graph $G(v)$. The tree $\cT^f$ is distributed like a critical Galton--Watson tree with reproduction law $\xi$. As the random graph $\mA_n^\omega$ is distributed like the conditioned graph $(\mA \mid |\mA|=n)$, we consider the version $(\cT_n^f, (G_n(v))_{v \in \cT_n^f})$ of $(\cT^f, (G(v))_{v \in \cT^f})$ conditioned on the event $|\mA|=n$. The tree $\cT_n^f$ has a random size, that by \cite[Lem. 6.3]{2015arXiv150402006S} satisfies a normal central limit theorem 
\begin{align}
	\label{eq:normal}
	\frac{|\cT_n^f| - n/(1 + \Ex{\zeta})}{\sqrt{n}} \convdis \cN(0, \sigma^2)
\end{align}
for some $\sigma>0$. For any $k \ge 1$ it holds that $(\cT_n^f \mid |\cT_n^f| = k)$ is distributed like the conditioned Galton--Watson tree $(\cT^f \mid |\cT^f|=k)$. Hence it follows from Aldous' invariance principle~\cite{MR1207226} for critical Galton--Watson trees whose offspring law has finite variance that the tree $\cT_n^f$ equipped with the law $\mu_n^f$ of a uniformly at random selected vertex converges in the rooted Gromov--Hausdorff--Prokhorov sense towards the CRT $(\CRT, d_{\CRT}, \mu)$ after rescaling the metric on $\cT_n^f$ by $c_f / \sqrt{n}$ for some positive constant $c_f$. So in order to prove Lemma~\ref{le:lemma} it suffices to verify that there is a positive constant $c_\omega$ such that Gromov--Hausdorff--Prokhorov distance of $(\mA_n^\omega, c_\omega n^{-1/2}d_{\mA_n^\omega}, \mu_n^\scA)$ and $(\cT_n^f, c_f  n^{-1/2} d_{\cT_n^f}, \mu_n^f)$ converges in probability to zero.

For each vertex $v \in \mA$ define the vertex $v^\circ$ by setting if $v^\circ = v$ if $v$ is a fixed-point, and letting $v^\circ$ be the unique fixed-point such that $v \in f(G(v^\circ))$ if $v$ is not a fixed-point. 
In the proof of the Gromov--Hausdorff scaling limit~\cite[Thm. 6.9]{2015arXiv150402006S} it was shown (in a more general context) that there is a  constant $c_\omega>0$ for which the Gromov--Hausdorff distance of the spaces $(\mA_n^\omega, c_\omega n^{-1/2}d_{\mA_n^\omega})$ and $(\cT_n^f, c_f  n^{-1/2} d_{\cT_n^f})$ converges in probability to zero by verifying
\begin{align*}
	n^{-1/2} \sup_{x,y \in \mA_n^\omega} |c_\omega d_{\mA_n}(x,y) - c_f d_{\cT_n^f}(x^\circ, y^\circ)| \convp 0.
\end{align*}
So if $u_n \in \mA_n^\omega$ gets drawn uniformly at random, this implies that the law $\mu_n^\circ$ of $u_n^\circ$ satisfies
\[
	d_{\mathrm{GHP}}\left(  (\mA_n^\omega, c_\omega n^{-1/2}d_{\mA_n^\omega}, \mu_n^\scA), (\cT_n^f, c_f  n^{-1/2} d_{\cT_n^f}, \mu_n^\circ) \right) \convp 0.
\]
Thus, in order to prove Lemma~\ref{le:lemma} it suffices to show that the Prokhorov distance of the uniform law $\mu_n^f$ and the law $\mu_n^\circ$ (which are Borel probability measures on the same random metric space $(\cT_n^f, c_f n^{-1/2} d_{\cT_n^f})$) satisfies
\begin{align}
	\label{eq:toshow}
	d_{\mathrm{P}}(\mu_n^f, \mu_n^\circ) \convp 0.
\end{align}

To this end, let $v_1, \ldots, v_{|\cT^f|}$ denote the depth-first-search ordered list of vertices of the tree $\cT^f$. Let $u \in \cT^f$ be drawn uniformly at random and let $i^\circ$ denote the unique index such that $u^\circ = v_{i^\circ}$ or $u^\circ \in F(G(v_{i^\circ}))$. For any $0 \le x \le 1$ it holds that
\begin{align}
	\label{eq:dahell}
	\Pr{ i^\circ \le x |\cT^f| \mid |\mA| = n } = \Ex{ n^{-1}\sum_{1 \le i \le x |\cT^f|} (1 + |F(G(v_i))|) \mid |\mA| = n}.
\end{align}
By~\eqref{eq:asymptotics} we know that $\Pr{|\mA|= n} = O(n^{3/2})$ and \eqref{eq:normal} yields that  $|\cT_n^f| \in n/(1+\Ex{\zeta}) \pm n^{2/3}$ with high probability. Let $(\zeta_i)_{i \ge 1}$ be independent copies of $\zeta$.  Hence the event \[
\cE_n = \left\{n^{-1}\sum_{1 \le i \le x |\cT^f|} (1 + |F(G(v_i))|) \notin x \pm n^{-2/3} \right\}
\]
satisfies
\begin{align}
	\label{eq:bojack}
	\Pr{ \cE_n \mid |\mA|= n} &\le o(1) +  O(n^{3/2}) \sum_{K \in n/(1 + \Ex{\zeta}) \pm n^{2/3}} \Pr{ \cE_n, |\cT^f| = K, |\mA| = n} \nonumber \\
	&\le o(1) + O(n^{3/2}) \sum_{K \in n/(1 + \Ex{\zeta}) \pm n^{2/3}} \Prb{ \sum_{1 \le i \le xK} (1 + \zeta_i) \notin nx \pm n^{1/3} }
\end{align}
As $\zeta$ has finite exponential moments, it follows by a well-known deviation inequality for one-dimensional random walk found in most books on the subject that there are constants $\delta, c>0$ such that for all $k \ge 1$, $t \ge 0$ and $0 \le \lambda \le \delta$ it holds that
\[
\Pr{|\zeta_1 + \ldots + \zeta_k - k \Ex{\zeta}| \ge t} \le 2 \exp(c k \lambda^2 - \lambda t).
\]
%Setting $t=k^{3/4}$ and $\lambda = k^{-2/3}$ it follows that
%\begin{align}
%\Pr{|\zeta_1 + \ldots + \zeta_k - k \Ex{\zeta}| \ge k^{2/3}} \le C\exp(-k/12)
%\end{align}
%for some constant $C>0$ that does not depend on $k$.
This means that in Inequality~\eqref{eq:bojack} we sum up polynomially many bounds that are uniformly  exponentially small, yielding that $\Pr{ \cE_n \mid |\mA|= n}$ tends to zero. Hence by dominated convergence it follows from~\eqref{eq:dahell} that the probability $\Pr{ i^\circ \le x |\cT^f| \mid |\mA| = n }$ tends to $x$ as $n$ becomes large. In other words, if $i_n^\circ$ denotes $i^\circ$ conditioned on the event $|\mA|=n$, then
\begin{align}
	\label{eq:prok}
	\frac{i^\circ_n}{|\cT_n^f|} \convdis X
\end{align}
for some random variable $X$ that is uniformly distributed on the unit interval $[0,1]$.
By Skorokhod's representation theorem~\cite[Thm. 3.3]{MR0310933} we may without loss of generality assume that all considered random variables are defined on the same probability space and $\frac{i^\circ_n}{|\cT_n^f|} \to X$ holds almost surely. For each $n$ partition the unit interval into $|\cT_n^f|$ disjoint equally long subintervals $(I_i)_{1 \le i \le |\cT_n^f|}$ and let $i^*_n$ denote the unique index with $X \in I_{i_n^*}$. Clearly $i^*_n$ is uniformly distributed over the set $\{1, \ldots, |\cT_n^f|\}$, so $v_{i^*_n}$ is a uniformly sampled vertex of $\cT_n^f$. Since $\frac{i^\circ_n}{|\cT_n^f|} \to X$ holds almost surely it follows that 
\begin{align}
	\label{eq:as}
	\left |\frac{i^\circ_n}{|\cT_n^f|} - \frac{i^*_n}{|\cT_n^f|} \right | \to 0
\end{align}
almost surely.

For each index $i$ let $h_i$ denote the height of the vertex $v_i$.  The  rescaled  height-process associated to the tree $\cT_n^f$ is the c\`adl\`ag function $\left( \frac{1}{\sqrt{|\cT_n^f|}} h_{ \lfloor t |\cT_n^f| \rfloor}, 0 \le t \le 1\right)$. It follows from~\eqref{eq:normal} and~\cite[Thm. 3.1]{MR1964956} that the rescaled height-process converges weakly  with respect to the Skorokhod topology towards a constant multiple $\kappa \me$ of Brownian excursion $\me = (\me_t, 0 \le t \le 1)$ of duration one. As Brownian excursion is almost surely continuous, this implies that the weak limit already holds with respect to the supremum norm. Tightness then implies that for any $\epsilon, \eta>0$ there is a $0 < \delta < 1$ such that
\begin{align}
	\label{eq:tight}
	\Prb{ |\cT_n^f|^{-1/2} \sup_{|s-t| < \delta}| h_{\lfloor s|\cT_n^f| \rfloor} -  h_{\lfloor t|\cT_n^f| \rfloor} | \ge \epsilon} < \eta
\end{align}
for large enough $n$. 

Suppose that  there are indices $1 \le i<j \le |\cT_n^f|$ such that  $|i-j| < \delta |\cT_n^f|$ and  $d_{\cT_n^f}(v_i, v_j) \ge 4 \epsilon \sqrt{|\cT_n^f|}$. Let $k$ be the index such that $v_k$ is the youngest common ancestor of $v_i$ and $v_j$. It holds that
\[
	d_{\cT_n^f}(v_i, v_j) = (h_i - h_k) + (h_j - h_k).
\]
Set $x := d_{\cT_n^f}(v_i,v_j)$. If $h_j - h_k \le x/3$, then it follows that $h_i - h_k \ge 2x/3$, and hence $h_i - h_j = (h_i - h_k) + (h_k - h_j) \ge x/3$. If on the other hand $h_j - h_k > x/3$, then the unique offspring $v_{k'}$ of $v_k$ that lies on the path between $v_k$ and $v_j$ satisfies $i < k' \le j$ and $h _j - h_{k'} \ge x/3 -1$. It follows that there are indices $i^*$ and $j^*$ such that 
\[
|i^* - j^*| < \delta |\cT_n^f| \quad \text{and} \quad |h_{i^*} - h_{j^*}| \ge \frac{4}{3} \epsilon \sqrt{|\cT_n^f|} -1.
\]
By Inequality~\eqref{eq:tight} it follows that for large enough $n$
\begin{align}
\label{eq:tight2}
\Prb{ |\cT_n^f|^{-1/2} \sup_{|s-t| < \delta} d_{\cT_n^f}(v_{\lfloor s|\cT_n^f|}, v_{\lfloor t|\cT_n^f|})   \ge 4 \epsilon} < \eta.
\end{align}

Now, let us put it all together. It follows from~\eqref{eq:as} and \eqref{eq:tight2} that
\[
\Prb{  d_{\cT_n^f}(v_{i^\circ_n}, v_{i^*_n})   \ge 4 \epsilon   \sqrt{|\cT_n^f|}} < 2 \eta
\]
for large enough $n$. As $\epsilon$ and $\eta$ where arbitrary, this implies
\[
|\cT_n^f|^{-1/2}d_{\cT_n^f}(v_{i^\circ_n}, v_{i^*_n}) \convp 0.
\]
The vertex $v_{i_n^*}$ is uniformly distributed among the nodes of the tree $\cT_n^f$, so \eqref{eq:toshow} follows and the proof of Lemma~\ref{le:lemma} is complete. 

\section{The scaling constant of unlabelled outerplanar graphs}
\label{sec:const}

\subsection{A general description}
The scaling constant $c_\omega$ of Theorem~\ref{te:scalinglimit} is identical to the scaling constant for unlabelled rooted graphs. Hence, by the general result \cite[Lem. 6.11, Proof of Thm. 6.9]{2015arXiv150402006S}, it is given by
\begin{align}
	\label{eq:comega}
	c_\omega = \frac{\sqrt{(1 + \Ex{\zeta})\Va{\xi}}}{2 \Ex{\eta}}
\end{align}
with $\xi$ and $\zeta$ the random variables given in  \eqref{eq:def}, and $\eta$ the distance between the marked points in a random unlabelled $(\scB'^\bullet)^\iota$-object $\mB'^\bullet$ that follows a Boltzmann distribution with parameter $\tilde{\scA}^\omega(\rho_\scA)$. That is, $\mB'^\bullet$ is equal to a random unlabelled $(\scB'^\bullet)^\iota$-object $B$ with probability $\tilde{\scA}^\omega(\rho_\scA)^{|B|} \iota(B) / (\tilde{\scB}'^\bullet)^\iota(\tilde{\scA}^\omega(\rho_\scA))$.  Note that $B$ is a graph having an inner root (also called the $*$-vertex) that does not contribute to the total size $|B|$, and an outer root that does contribute to the total size and is required to not coincide with the inner root. It follows from the expression for the generating function of $(\xi, \zeta)$ in Equation~\eqref{eq:def} that
	\begin{align}
	\label{eq:almost}
	\Va{\xi}&=E_{uu} (\rho_\scA, \tilde{\scA}^\omega(\rho_\scA))\tilde{\scA}^\omega(\rho_\scA) \quad \text{and} \quad 
\Ex{\zeta}= E_z(\rho_\scA, \tilde{\scA}^\omega(\rho_\scA)) \rho_\scA / \tilde{\scA}^\omega(\rho_\scA) -1,
\end{align}
with the power series $E(z,u)$ defined in Equation \eqref{eq:EE} (and $E_z$ and $E_{uu}$ denoting partial derivatives). The main challenge is to compute $\Ex{\eta}$.

\subsection{Unlabelled outerplanar graphs}
We are now going to derive  numeric approximations of the constant $c_\omega$ in \eqref{eq:comega} for the $\omega$-weighting $\omega_\scO$ that corresponds to uniform unlabelled outerplanar graphs. It was shown in \cite[Thm. 2.16]{vigerske}, \cite[Cor. 3.6]{MR2350456} that
\begin{align}
\label{eq:fm}
Z_{(\scB')^{\iota_\scO}}(s_1, s_2, \ldots) =\frac{1}{8}\left(1+s_1-\sqrt{s_1^2-6s_1+1}\right)
+\frac{1}{8s_2^2}(s_1+s_2)\left(1-3s_2-\sqrt{s_2^2-6s_2+1}\right),
\end{align}
with $\iota_\scO$ denoting the special case of the $\iota$-weighting for outerplanar graphs. The formula was obtained in the cited sources by computing $Z_{(\scB)^{\iota_\cO}}$ and then using $Z_{(\scB')^{\iota_\scO}} = \frac{\partial}{\partial s_1} Z_{(\scB)^{\iota_\cO}}$.  We are going to derive \eqref{eq:fm} in a direct way that will be convenient later on for the computation of $c_\omega$.

\begin{wrapfigure}{r}{0.4\textwidth}
	\vspace{-10pt}
	\begin{center}
		\includegraphics[width=0.4\textwidth]{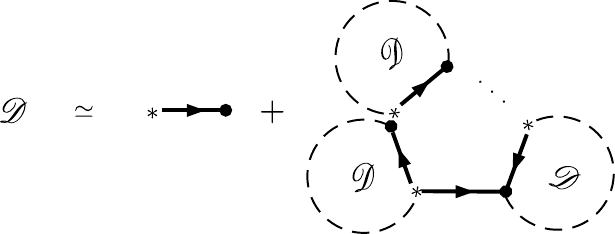}
	\end{center}
	\caption{Decomposition of dissections}
	\label{fi:dis}
\end{wrapfigure}

 To this end, consider the species $\scD$ of dissections of polygons where one edge lying on the frontier of the outer face is marked and oriented, and where the origin of the root-edge is a $*$-vertex that does not contribute to the total size.  Note that $\scD$ is, like any corner-rooted planar map, \emph{asymmetric}, meaning that any permutation of the non-$*$-vertices that leaves the structure invariant must be the identity. The smallest $\scD$-object has size $1$ and consists of a single oriented root-edge. Any larger $\scD$-object may be decomposed in a unique way into an ordered list of at least two $\scD$-objects as illustrated in Figure~\ref{fi:dis}, yielding 
\begin{align}
	\label{eq:disser}
	\scD \simeq \scX +  \sum_{k \ge 2} \scD^k \qquad \text{and} \qquad Z_{\scD} = s_1 + Z_{\scD}^2/(1 - Z_\scD) = \frac{1}{4}\left(1+s_1-\sqrt{s_1^2-6s_1+1}\right).
\end{align}
Any \emph{labelled} $(\scB')^{\iota_\cO}$-object (with positive weight) with size at least $2$  has a unique Hamilton cycle that may be oriented in two different ways, yielding
\begin{align}
	\label{eq:z0}
	Z_0 := Z_{(\scB')^{\iota_\scO}}(s_1, 0, 0, \ldots) = \frac{1}{2}(Z_{\scD} + s_1).
\end{align}
It remains to sum up the weight-monomials of symmetries that are different from the identity. Any automorphism of a \emph{labelled} $(\scB')^{\iota_\cO}$-object is also an automorphism of its Hamilton cycle and hence an element of a dihedral group, meaning it is composed out of rotations and reflections along axis passing through the "center" of the Hamilton  cycle.  But the $*$-vertex is always required to be fixed, hence the only possible non-trivial automorphism is a reflection along the axis that passes through the $*$-vertex and the "center".
This also means that any such graph has only two proper embeddings into the plane, so it makes sense to define the root-face as the unique inner face adjacent to the $*$-root.

\begin{wrapfigure}{l}{0.3\textwidth}
	\vspace{-10pt}
	\begin{center}
		\includegraphics[width=0.25\textwidth]{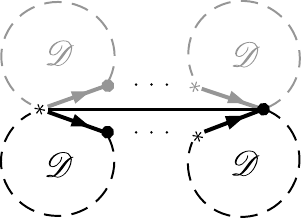}
	\end{center}
	\caption{Symmetric blocks}
	\label{fi:axis1}
\end{wrapfigure}

There are two different cases, depending on whether the axis leaves the Hamilton cycle through the middle of an edge, or through a non-$*$-vertex. The latter case may be partitioned into two subcases, depending on whether there is an edge between the $*$-vertex and this second vertex or not. If this edge is present, then the graph is uniquely determined by the ordered list of $\scD$-objects encountered along one half of the root-face, see Figure~\ref{fi:axis1}. This list must have length at least two as we do not allow multi-edges. All atoms belong to $2$-cycles of the reflection, except for the destination of the root-edge in the last element of the list, since this vertex must be a fixed-point. Moreover, any such graph with $n$ non-$*$-vertices has precisely $n!/2$ labellings, so the sum of weight-monomials of all symmetries of such dissections is given by
\begin{align}
	\label{eq:z1}
	Z_1 := \frac{1}{2}\left( \sum_{k \ge 2} \scD(s_2)^k \right)\frac{s_1}{s_2} = \frac{\scD(s_2)^2 \frac{s_1}{s_2}}{2(1 - \scD(s_2))} = \frac{1}{2}\left(\scD(s_2) \frac{s_1}{s_2} - s_1\right).
\end{align}
with $\scD(z) := Z_{\scD}(z,0,0,\ldots)$.

\begin{wrapfigure}{r}{0.3\textwidth}
%	\vspace{-10pt}
	\begin{center}
		\includegraphics[width=0.25\textwidth]{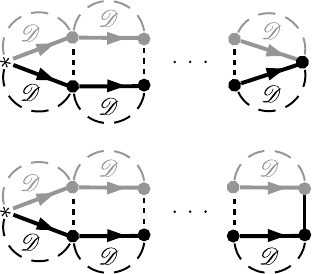}
	\end{center}
	\caption{Symmetric derived blocks with no edge along the axis of symmetry}
	\label{fi:axis2}
\end{wrapfigure}

If there is no edge along the axis of symmetry, then the endpoints of the root-edges of any corresponding pair of identical dissections may be joined by chord. Compare with the upper part of Figure~\ref{fi:axis2}, where these potential chords are indicated by dashed vertical straight lines. Again, any unlabelled graph of this form with $n$ non-$*$-vertices has precisely $n!/2$ labellings. So the cycle index sums for the case where the axis passes through two vertices is given by
\begin{align}
\label{eq:z2}
Z_2 := \frac{1}{2}\left( \sum_{k \ge 1} (2\scD(s_2))^k \right) \scD(s_2)\frac{s_1}{s_2} = \frac{\scD(s_2)^2 \frac{s_1}{s_2}}{1 - 2\scD(s_2)}.
\end{align}
Here the factor $2$ in front of $Z_\scD$ is due to the two options that the chord is present or not. Similarly, the case where the axis leaves the Hamilton cycle through an edge contributes the cycle index sum
\begin{align}
\label{eq:z3}
Z_3 := \frac{1}{2}\left( \sum_{k \ge 0}  (2\scD(s_2))^k \right) \scD(s_2) = \frac{\scD(s_2)}{2(1 - 2\scD(s_2))}.
\end{align}
Summing up, we obtain
\begin{align*}
	Z_{(\scB')^{\iota_\scO}} = Z_0 + \ldots + Z_3 &= \frac{1}{2}\left( \scD(s_1) + \scD(s_2) \frac{s_1}{s_2}  + \frac{2\scD(s_2)^2 \frac{s_1}{s_2}+\scD(s_2)}{1 - 2\scD(s_2)}  \right).
\end{align*}
Using a computer algebra system, we may verify that this is identical to the expression given in Equation~\eqref{eq:fm}.

It follows from this description that the cycle index sum  $Z_{(\scB'^\bullet)^{\iota_\scO}}$ of bi-pointed two-connected outerplanar graphs is given by
\begin{align}
\label{eq:dec}
Z_{(\scB'^\bullet)^{\iota_\scO}} = Z_0^\bullet + Z_1^\bullet + Z_2^\bullet = \frac{1}{2} \left( \scD^\bullet(s_1) + s_1 \right) + Z_1 + Z_2
\end{align}
with $F^\bullet = s_1 \frac{\partial F}{\partial s_1}$ for any series $F$. But what is important to compute $\Ex{\eta}$ is not just the series but its combinatorial interpretation, which is why we made the effort to derive these equations, rather than just recalling~\eqref{eq:fm}. For ease of notation, let us set $a= \tilde{\scA}^\omega(\rho_\scA)$ and $b= \tilde{\scA}^\omega(\rho_\scA^2)$. In order to sample the Boltzmann distributed random graph $\mB'^\bullet$, we may sample a symmetry from $\bigcup_{k \ge 1} \Sym(\scB'^\bullet)[k]$ according to an $(a,b)$-Boltzmann distribution for $\Sym(\scB'^\bullet)$ (that is, any symmetry $(B, \sigma)$ is attained with probability $ \frac{\iota_\scO(B)}{|B|!} a^{\sigma_1} b^{\sigma_2} / Z_{(\scB'^\bullet)^{\iota_\scO}}(a,b)$)  and then forget about the automorphism. The decomposition that lead us to~\eqref{eq:dec} allows us to do sample this random symmetry in a multi-step process. 

\begin{enumerate}
	\item First we select $I \in \{0,1,2\}$ with probability $\Pr{I = i} = Z_i^\bullet(a,b) / Z_{(\scB'^\bullet)^{\iota_\scO}}(a,b)$. 
	\item If $I=0$, we sample a random labelled graph $\mB^0$ from $\bigcup_{k \ge 0} (\scB'^\bullet)^{\iota_\scO}[k]$ with probability given by  $ \Pr{\mB^0 = B} = \frac{\iota_\scO(|B|)}{|B|!}a^{|B|} / (\frac{1}{2} \scD^\bullet(a) + a)$ and equip it with the trivial automorphism.
	\item If $I\in \{1,2\}$ then $Z_i^\bullet = Z_i$, as any symmetry constructed to form $Z_i$ has a unique fixed atom. We select a random symmetry $(\mB^I, \sigma(I))$ among all the symmetries constructed to sum up the cycle index sum $Z_I$ in \eqref{eq:z1} (in case $I=1$)  or \eqref{eq:z2} (in case $I=2$), respectively, with probability
	\[
		\Pr{(\mB^I, \sigma(I)) = (B, \sigma)} = \frac{\iota_\scO(B)}{|B|!} a^{\sigma_1(I)} b^{\sigma_2(I)} / Z_I(a,b).
	\]
\end{enumerate}	
Let $\eta_I$ denote the distance between the $*$-vertex and the marked root in the random graph $\mB^I$.  This yields
\begin{align}
	\label{eq:buyka}
	\Ex{\eta} = \frac{1}{Z_{(\scB'^\bullet)^{\iota_\scO}}(a,b)}\left( \frac{1}{2}(a\scD'(a) + a) \Ex{\eta_0} + Z_1(a,b)\Ex{\eta_1} + Z_2(a,b) \Ex{\eta_2} \right).
\end{align}
Recall that we constructed the symmetries for $Z_1$ as in Figure~\ref{fi:axis1} so that there is always an edge between the $*$-vertex and the marked root. Hence $\eta_1=1$ is constant and so
\begin{align}
	\Ex{\eta_1} = 1.
\end{align} As for the case $I=2$ (see the upper half of Figure~\ref{fi:axis2}), the distance $\eta_2$ is equal to the number of dissections attached to the root-face along one half of the symmetry axis. By the arguments that lead to Equation~\eqref{eq:z2}, it follows that $\eta \ge 2$ and for  any $k \ge 2$ it holds $\Pr{\eta_2 = k} =  (2 \scD(b))^k \frac{a}{4b} / Z_2(a,b)$. This yields
\begin{align}
	\Ex{\eta_2} = Z_2(a,b)^{-1} \sum_{k \ge 2} k (2 \scD(b))^k \frac{a}{4b}  =  \frac{2 - 2 \scD(b)}{1 - 2\scD(b)}.
\end{align}
Finally, consider the case $I=0$. Here we actually treat random labelled graphs, and we may build upon results obtained in this setting. It follows from the arguments that lead to Equation~\eqref{eq:z0} that 
\begin{align}
\Ex{\eta_0} = \frac{a}{\scD^\bullet(a) + a} + \frac{\scD^\bullet(a)}{\scD^\bullet(a) + a} \Ex{\eta_0'}
\end{align}
with $\scD^\bullet(z) = z \scD'(z)$ and $\eta_0'$ the distance between the $*$-vertex and the marked root in a random marked dissection from the class $\scD^\bullet$ that assumes any marked dissection $D^\bullet$ with probability $ \frac{a^{|D^\bullet|}}{|D^\bullet|!} / \scD^\bullet(a)$. Let us set $w := \scD(a)$. It follows from the proof of \cite[Lem. 8.9]{inannals}, where calculations for marked dissection where carried out with different parameters, that there exist numbers $R,S>0$ such that
{
	\small
\begin{align*}
\begin{pmatrix}
2{\w}^{4} -4{\w}^{3}+3\w-1  & -{\w}^{3}+{\w}^{2} & {\w}^{3}-2{
	\w}^{2}+\w  \\
-{\w
}^{3}+{\w}^{2} & 2{\w}^{4}-4{\w}^{3}+
3\w-1 & {\w}^{3}-2{\w}^{2}+\w\\
-{\w}^{2}+\w & -{\w}^{2}+\w & 2{\w}^{4}-4{\w}^{3}+{\w}^{2}+2 \w-1 
\end{pmatrix}
\begin{pmatrix}
\Ex{\eta_0'} \\ R \\ S 
\end{pmatrix}
=\begin{pmatrix}
2{\w}^{4}-4{\w}^{3}-{\w}^
{2}+3\w-1 \\ -\w \\ -{\w}^{2}
\end{pmatrix}.
\end{align*}
}
\noindent
This inhomogeneous system of linear equations (with the indeterminates $\Ex{\eta_0'},R,S$) had a unique solution for the parameter considered in \cite[Lem. 8.9]{inannals}, but we still have to check if this the case in our setting.  The growth constant for unlabelled outerplanar graphs was approximated in  \cite[Sec. 3.1.3]{vigerske}, \cite[Sec. 4.2]{MR2350456} by numerically solving truncated systems of equations, yielding $\rho_\scA \approx 0.1332694$, $a \approx 0.1707560$, $b \approx 0.0180940$,  $E_{uu}(\rho_{\scA}, a) \approx 549.359$ and $E_z(\rho_\scA, a) \approx 1.34975$. See \cite[Sec. 3.1.3]{vigerske} for preciser estimates, that we used to carry out all following calculations.
%We implemented and numerically solved these truncated systems (see \cite[Sec. 3.1.3]{vigerske} for details on how to do that; we used truncation order $m=50$ in the notation of \cite{vigerske}) to obtain precise estimates of these constants, that we used to carry out all following calculations. 
The determinate of the matrix in the system of linear equations evaluates to $\approx -0.00805 \ne 0$. Hence there is a unique solution of the associated inhomogeneous system, yielding $\Ex{\eta_0'} = {\frac {8{\w}^{4}-16{\w}^{3}+4\w-1}{ \left( 4{\w}^{3}-6{\w}^{2}-2\w+1 \right)  \left( 2\w-1
		\right) }} \approx 5.435858$. This allows us to evaluate Equation~\eqref{eq:buyka}, yielding $\Ex{\eta} \approx 5.038561$. Using Equation~\eqref{eq:almost} we obtain $\Ex{\zeta} \approx 0.0534353$ and $\Va{\xi} \approx 93.80631$. Hence Equation~\eqref{eq:comega} evaluates to
\begin{align}
	c_{\omega_\scO} \approx 0.9864689,
\end{align}
as we stated in  Proposition~\ref{pro:proconst}.

%\textcolor{red}{approximation for $b$ was done by evaluating $C^{[m]}(.13326943266744680944^2)$;  $c_k \rho^{2k}$ has order $k^{-3/2} \rho^k$, so the error has order $O(\rho^k)$ which is exponentially small; }

\appendix
\section{Species theory}
\label{sec:species}

In this appendix, we summarize some  aspects of combinatorial species used in our proofs following \cite{MR633783,MR1629341,MR2810913}. A thorough introduction to the subject is beyond the scope of this paper and we refer the reader to these sources.

%The present section is dedicated to recalling the required combinatorial background to prepare for the proof of the main results. 

%A comprehensive and elegant account on the topic is given in the pioneering paper by Joyal~\cite{MR633783}. There the author assumes that the reader has some knowledge of both the French language and the language of category theory, but we are not going to assume either. There are of course excellent reasons for using category theory in this context, for example due to the necessity of monoidal categories to rigorously state associative laws \cite[Ch. 7]{MR633783}, but avoiding this terminology allows us to address a broader audience.

% An extensive and elegant account on the topic is given in the pioneering paper by Joyal~\cite{MR633783}. There the author assumes that the reader has some knowledge of  language of category theory. There are excellent reasons for using category theory in this context, for example due to the necessity of monoidal categories to rigorously state associative laws \cite[Ch. 7]{MR633783}, but we are going to avoid this terminology to address a broad audience.

\subsection{Weighted combinatorial species}
%Formally, a species of combinatorial structures $\scF^\omega$  with non-negative weights may be defined as a functor from the category of finite sets and bijections to the category of weighted finite sets. 
We are going to define \emph{combinatorial species with weights} in the set $\ndR_{\ge 0}$ of non-negative real numbers. Such an object $\scF^\omega$ may be described as follows.
For each finite set $U$ the species $\scF^\omega$ produces a finite set $\scF[U]$ of \emph{$\scF$-structures} and a \emph{weight-map}
$
\omega_U: \scF[U] \to \ndR_{\ge 0}.
$
Furthermore, for any bijection $\sigma: U \to V$ between finite sets the species $\scF^\omega$ produces a \emph{transport function}
$
\scF[\sigma]: \scF[U] \to \scF[V],
$
which must preserve the $\omega$-weights. In other words, the diagram
\[
\xymatrix{ \scF[U]  \ar[r]^{\scF[\sigma]} \ar[dr]^{\omega_U} 
	&\scF[V]\ar[d]^{\omega_V}\\
	&\ndR_{\ge 0}}
\]
is required to commute. The bijections produced by a species are subject to functoriality conditions: the identity map $\text{id}_U$ on a finite set $U$ gets mapped to the identity map $\scF[\text{id}_U] = \text{id}_{\scF[U]}$ on the set $\scF[U]$. For any bijections $\sigma: U \to V$ and $\gamma: V \to W$ the diagram
\[
\xymatrix{ \scF[U]  \ar[r]^{\scF[\sigma]} \ar[dr]^{\scF[\gamma \sigma]} 
	&\scF[V]\ar[d]^{\scF[\gamma]}\\
	&\scF[W]}
\]
must commute. We further assume that $\scF[U] \cap \scF[V] = \emptyset$ whenever $U \ne V$. This is not much of a restriction, as we may always replace $\scF[U]$ by $\{U\} \times \scF[U]$ for all sets $U$, to make sure that it is satisfied. A reader familiar with category theory may without doubt recognize that combinatorial species are endo-functors of the groupoid of finite weighted sets and weight-preserving bijections. In particular, any concerns regarding set-theoretic aspects of the definition of combinatorial species may be dispersed by consulting any book on category theory, in particular the standard treaty~\cite{MR1712872}.

Two weighted species $\scF^\omega$ and $\scH^\nu$ are \emph{isomorphic}, denoted by $\scF^\omega \simeq \scH^\nu$, if there is a family of weight-preserving bijections $(\alpha_U: \scF[U] \to \scH[U])_U$ with $U$ ranging over all finite sets, such the following diagram commutes for each  bijection  $\sigma: U \to V$ of finite sets.
\[
\xymatrix{ \scF[U] \ar[d]^{\alpha_U} \ar[r]^{\scF[\sigma]} &\scF[V]\ar[d]^{\alpha_V}\\
	\scH[U] \ar[r]^{\scG[\sigma]} 		    &\scH[V]}
\]

We say $\scH^\nu$ is a \emph{subspecies} of $\scF^\omega$, if for each finite set $U$, any bijection $\sigma: U \to V$ and each $\scH$-object $H \in \scH^\nu$ it holds that $\scH[U] \subset \scF[U]$, $\scH[\sigma](H) = \scF[\sigma](H)$ and $\nu_U(H)  = \omega_U(H)$. We denote this by $\scH^\nu \subset \scF^\omega$. By abuse of notation, we will usually denote the weighting on both $\scH$ and $\scF$ by $\omega$.

It will be convenient to simply write $\omega(F)$ instead of $\omega_U(F)$ for the weight of a structure $F \in \scF[U]$. If no weighting is specified explicitly for a species $\scF$, we assume that all structures receive weight $1$. We refer to the set $U$ as the set of \emph{labels} or \emph{atoms} of the structure. For any $\scF$-object $F \in \scF[U]$ we let
$
|F| := |U| \ge 0
$
denote its \emph{size}. 

\subsection{Ordinary generating series and cycle index sums}
Given a finite set $U$, the symmetric group $\scS_U$ operates on the set $U$ via
\[
\sigma.F = \scF[\sigma](F)
\]
for all $F \in \scF[U]$ and $\sigma \in \mathscr{S}_U$. Any bijection $\sigma$ with $\sigma.F = F$ is termed an \emph{automorphism} of $F$.  All $\scF$-objects of an orbit $\tilde{F}$ have the same size and same $\omega$-weight, which we denote by $|\tilde{F}|$ and $\omega(\tilde{F})$. This yields the weighted set $\tilde{\scF}[U]$ of orbits under this operation.
Formally, an \emph{unlabelled} $\scF$-object is defined as an isomorphism class of $\scF$-objects. We may also identify the unlabelled objects of a given size $n$ with the orbits of the action of the symmetric group on any $n$-sized set.  By abuse of notation, we treat unlabelled objects as if they were regular $\scF$-objects. The power series
\[
\tilde{\scF}^\omega(z) = \sum_{\tilde{F}} \omega(\tilde{F}) z^{|\tilde{F}|}
\]
is the \emph{ordinary generating series} of the species. Here the index ranges over all unlabelled $\scF$-objects.

To any species $\scF$ we may associate the corresponding functor $\Sym(\scF)$ of \emph{$\scF$-symmetries} such that
\[
\Sym(\scF)[U] = \{ (F, \sigma) \mid F \in \scF[U], \sigma \in \mathscr{S}_U, \sigma.F = F\}.
\]
In other words, a symmetry is a pair of an $\scF$-object and an automorphism. 
The transport along a bijection $\gamma: U \to V$ is given by
\[
\Sym(\scF)[\gamma](F, \sigma) = (\scF[\gamma](F), \gamma \sigma \gamma^{-1}).
\]
For any permutation $\sigma$ we let $\sigma_i$ denote its number of $i$-cycles. In particular, $\sigma_1$ counts the number of fixpoints. The \emph{cycle index series} of a species $\scF^\omega$ is defined as the formal power series
\[
Z_{\scF^\omega}(s_1, s_2, \ldots) = \sum_{k \ge 0} \sum_{(F, \sigma) \in \Sym(\scF)[k]} \frac{\omega(F)}{k!} s_1^{\sigma_1} \cdots s_k^{\sigma_k}
\]
in countably infinitely many indeterminates $(s_i)_{i \ge 1}$.  Consider symmetries is useful, as  it provides a way of counting orbits:

\begin{lemma}
	\label{le:relation1}
	For any finite set $U$ with $n$ elements and any unlabelled $\scF$-object $\tilde{F} \in \tilde{\scF}[U]$ there are precisely $n!$ many symmetries $(F, \sigma) \in \Sym(\scF)[U]$ such that $F$ belongs to the orbit $\tilde{F}$. Hence there is a weight-preserving $1$ to $n!$ relation between $\tilde{\scF}[U]$ and $\Sym(\scF)[U]$. Consequently:
	\[
	\tilde{\scF}^\omega(z) = Z_{\scF^\omega}(z, z^2, z^3, \ldots).
	\]
\end{lemma}
This standard result is explicit in Bergeron, Labelle and Leroux \cite[Ch. 2.3]{MR1629341}.% and shows how the ordinary generating series and the cycle index sum of a species are related.

\subsection{The cycle pointing operator}
For each finite set $U$ and permutation $\sigma \in \scS_U$, the generated subgroup $<\sigma>\subset \scS_U$ operates canonically on~$U$. The restriction of $\sigma$ to any single orbit of this operation is termed a \emph{cycle} of $\sigma$. For any cycle $\tau$ we let its \emph{length} $|\tau|$ be the number of elements of the corresponding orbit.
The \emph{cycle pointed species} $(\scF^\circ)^\omega$ associated to a species $\scF^\omega$ is defined as follows. For each finite set $U$, the elements of the set $\scF^\circ[U]$ are all pairs $(F, \tau)$ of an $\scF$-structure $F$ and a cyclic permutation $\tau$ of some subset of $U$ such that there is at least one automorphism $\sigma \in \scS_U$ of $F$ having $\tau$ as one of its disjoint cycles. (Here we allow the case where $\tau$ is just a fixed-point of $\sigma$.) The transport along a bijection $\gamma: U \to V$ is defined by
\[
\scF^\circ[\gamma](F, \tau) = (\scF[\gamma](F), \gamma \tau \gamma^{-1}).
\]
The weighting of the cycle pointed version is inherited from the original species by
\[
\omega(F, \tau) = \omega(F).
\]
The idea behind the cycle pointing operator is that it provides a way for counting objects up to symmetry.
\begin{lemma}
	\label{le:relation2}
	For any finite set $U$ with $n$ vertices there is a weight-preserving $1$ to $n$ correspondence between the set $\tilde{\scF}[U]$ of orbits of $\scF$-objects and the set $\tilde{\scF}^\circ[U]$ of orbits of $\scF^\circ$-objects.
\end{lemma}
This result has been proven in \cite[Lemma 4]{MR2810913} in the context of species without weightings, and the generalization to the weighted context is straight-forward. Lemma~\ref{le:relation2} shows that there is no difference in sampling a random $n$-sized unlabelled object with probability proportional to its weight from $\scF$ and $\scF^\circ$.

Any subspecies $\scH^\nu \subset (\scF^\circ)^\omega$ is termed \emph{cycle-pointed} as well. A natural example is the subspecies $(\scF^{\circledast})^\omega  \subset (\scF^\circ)^\omega$ of \emph{symmetrically} cycle-pointed objects for which the length of the marked cycle of each object is required to be at least $2$. For any finite set $U$ we let $\RSym(\scH)[U]$ denote the set of all tuples $(H, \sigma, \tau, v)$ with $(H, \tau) \in \scH[U]$, $\sigma$ an automorphism of $H$ having $\tau$ as one of its disjoint cycles, and $v \in U$ an atom of the cycle $\tau$. In order to keep track of the length of the marked cycle, cycle-pointed species receive an extended version of the cycle index sum
\[
\bar{Z}_{\scH^\nu}(s_1, t_1; s_2, t_2; \ldots) = \sum_{k \ge 0} \frac{1}{k!} \sum_{ (H, \sigma, \tau, v) \in \RSym(\scH)[k]} \nu(H) \frac{t_{|\tau|}}{s_{|\tau|}} s_1^{\sigma_1} s_2^{\sigma_2} \ldots s_k^{\sigma_k}.
\]
Lemma~\ref{le:relation2} readily implies that
\[
	\tilde{\scH}^\nu(z) = \bar{Z}_{\scH^\nu}(z,z; z^2, z^2; \ldots).
\]

\begin{comment}
\subsection{P\'olya-Boltzmann distributions}
\label{sec:polya}
Let $\scF^\omega$ be a weighted species and $x_1, x_2 \ldots \ge 0$  parameters with \[0 < Z_{\scF^\omega}(x_1, x_2, \ldots) < \infty.\]
The associated \emph{P\'olya-Boltzmann distribution} $\mathbb{P}_{Z_{\scF^\omega}, (x_i)_{i \ge 1}}$ is a probability measure on the set
\[
\bigsqcup_{n \ge 0} \Sym(\scF)[n]
\]
given by
\[
\mathbb{P}_{Z_{\scF^\omega}, (x_i)_{i \ge 1}}(F, \sigma) = Z_{\scF^\omega}(x_1, x_2, \ldots)^{-1} \frac{\omega(F)}{|F|!} \prod_{i \ge 1} x_i^{\sigma_i}.
\]
Let $x>0$ be a parameter with \[
0 < \tilde{\scF}^\omega(x) <\infty
\] and  $n \ge 0$ an integer  with $[z^n]\tilde{\scF}^\omega(z) >0$. If $(\cF, \sigma)$ is a random symmetry following a $\mathbb{P}_{Z_{\scF^\omega}, (x, x^2, x^3, \ldots)}$-distribution, then the orbit corresponding to the conditioned $\scF$-object $(\cF \mid |\cF| = n)$ gets drawn from the set $\tilde{\scF}[n]$ with probability proportional to its weight. In a way, this is analogous to the fact that simply generated planted plane trees with an analytic weight-sequence are distributed like Galton-Watson trees conditioned on having a specific size.
\end{comment}

\subsection{Further operators}

There are various standard ways to combine given species (cycle-pointed or not) to form new ones. We briefly recall some notation and relevant facts, but refer the reader to the literature \cite[Section 2.3]{MR1629341} and \cite{MR2810913} for a thorough description of these constructions. Throughout we let $\scF^\omega$ and $\scG^\nu$ denote weighted species.

\subsubsection{Constructions without cycle-pointing}

 If  $\scG^\nu[\emptyset] = \emptyset$ then we may form the \emph{composition} or \emph{substitution} $\scF^\omega \circ \scG^\nu$. It is a weighted species that describes partitions of finite sets, where each partition class is endowed with a $\scG$-structure, and the collection of partition classes carries an $\scF$-structure. The weight of such a  composite structure is the product of weights of its $\scF^\omega$-structure and $\scG^\nu$-structures. The cycle index sum of the substitution is given by 
\begin{align*}
%\label{eq:zcomp}
Z_{\scF^\omega \circ \scG^\nu}(s_1, s_2, \ldots) = Z_{\scF^\omega}(Z_{\scG^\nu}(s_1, s_2, \ldots), Z_{\scG^{\nu^{2}}}(s_2, s_4, \ldots), Z_{\scG^{\nu^{3}}}(s_3, s_6, \ldots), \ldots ).
\end{align*}  
Here $\nu^i$ denotes the weighting that assigns to each $\scG$-object $G$ the weight $\nu^i(G) = \nu(G)^i$. See for example \cite[Theorem 3 and Section 6]{MR633783} or \cite[Proposition 11 of Section 2.3]{MR1629341} for  details.

The \emph{product} $\scF^\omega \cdot \scG^\nu$ describes ordered pairs of an $\scF^\omega$ and an $\scG^\nu$ structure. The weight of such a structure is the product of weights of its components. The cycle index sum of the product satisfies $
Z_{\scF^\omega \cdot \scG^\mu} = Z_{\scF^\omega} Z_{\scG^\nu}
$.  The \emph{sum} $\scF^\omega + \scG^\nu$ describes the disjoint union of the two species, that canonically extends to weights and transport functions.
The cycle index sum of the sum satisfies $Z_{\scF^\omega + \scG^\nu} = Z_{\scF^\omega} + Z_{\scG^\nu}$. It is straight-forward to generalize this concept to sums of countably many species subject to the summability constraint that in total only finitely many unlabelled objects of any fixed size are present. The \emph{derived} species $(\scF')^\omega$ describes $\scF^\omega$-objects where one atom is marked and no longer contributes to the total size. Its cycle index series is given by $Z_{(\scF')^\omega}(s_1,s_2, \ldots) = \frac{\partial}{\partial s_1} Z_{\scF^\omega}(s_1,s_2, \ldots)$. Similar to the derived species, the \emph{pointed} species $\scF^\bullet$ is given by $\scF^\bullet = \scF' \cdot \scX$ with $\scX$ is the species having a single object of size $1$ and weight $1$.

\subsubsection{Constructions for cycle-pointed species}

It the species $\scF^\omega\subset (\scH^\circ)^\omega$ is cycle-pointed and $\scG[\emptyset]=\emptyset$ we may form the \emph{cycle-pointed substitution} $\scF^\omega \circledcirc \scG^\nu$ as follows. Given an $(\scH \circ \scG)^\circ$-structure $((H, (G_Q)_{Q \in \pi}), \tau)$, there must be an automorphism $\sigma$ having $\tau$ as one of its cycles. Let 
$
\bar{\sigma}: \pi \to \pi,  Q \mapsto \sigma(Q)
$
denote the corresponding induced map on the partition. For each atom $v$ of the cycle $\tau$ let $Q(v) \in \pi$ denote the unique partition class to which it belongs. Clearly it must hold that
$
\bar{\sigma}(Q(v)) = Q(\tau(v)).
$
Hence $\bar{\sigma}$ restricted to the set $\{Q(v) \mid v \in \tau\}$ forms a cycle $\bar{\tau}$. This makes $(H, \bar{\tau}) \in \scH^\circ[\pi]$ a cycle-pointed $\scH$-structure, that is called the \emph{core structure}. If the core structure belongs to the subset $\scF[\pi] \subset \scH^\circ[\pi]$, then we say $((H, (G_Q)_{Q \in \pi}), \tau)$ belongs to the \emph{cycle-pointed substitution} of $\scF^\omega$ with $\scG^\nu$. This defines a subspecies 
$
\scF^\omega \circledcirc \scG^\nu \subset (\scH^\omega \circ \scG^\nu)^\circ
$, and the weighting on $\scF^\omega \circledcirc \scG^\nu$ is inherited from $(\scH^\omega \circ \scG^\nu)^\circ$.
By \cite[Prop. 18]{MR2810913}, the extended cycle index sum of the cycle-pointed substitution is given by
\begin{align}
\bar{Z}_{\scF^\omega \circledcirc \scG^\nu}(s_1, t_1; s_2, t_2; \ldots) = \bar{Z}_{\scF^\omega}(g_1, \bar{g}_1; g_2, \bar{g}_2; \ldots)
\end{align}
with
$
g_i = Z_{\scG^{\nu^i}}(s_i,s_{2i}, s_{3i}, \ldots )
$
and
$
\bar{g}_i = \bar{Z}_{(\scG^\circ)^{\nu^i}}(s_i,t_i; s_{2i}, t_{2i}; s_{3i}, t_{3i}; \ldots ).
$
To be precise, \cite[Prop. 18]{MR2810913} states this equality for species without weights, but the generalization to the weighted context is straight-forward. 

With the species $\scF^\omega$ being cycle-pointed, the product $\scG^\nu  \cdot \scF^\omega $ may also be interpreted as a cycle-pointed species $\scG^\nu \star \scF^\omega$, since the marked cycle of the $\scF$-structure is also a cycle of some automorphism of the $\scG \cdot \scF$-structure. The corresponding extended cycle index sum is given by
$
\bar{Z}_{\scG^\nu \star  \scF^\omega } = Z_{\scG^\nu} \bar{Z}_{\scF^\omega}.
$
Likewise, if $\scF^\omega$ and $\scG^\nu$ are both cycle pointed, then so is their sum $\scF^\omega + \scG^\nu$ and the extended cycle index sum satisfies
$\bar{Z}_{\scF^\omega + \scG^\nu} = \bar{Z}_{\scF^\omega} + \bar{Z}_{\scG^\nu}$. See~\cite{MR2810913} for details.

\subsection{Associative laws}
	There are natural associative laws for the sum, product and substitution operations of the form
	\begin{align}
	\label{eq:iso}
	(\scF^\omega \mu \scG^\nu) \mu \scH^\kappa \simeq \scF^\omega \mu (\scG^\nu \mu \scH^\kappa)
	\end{align}
	for $\mu \in \{+,\cdot, \circ\}$, that ensure that regardless how we put the parentheses, the results are always isomorphic as species. Even more, there are natural choices of isomorphisms in \eqref{eq:iso} such that regardless in which order we successively apply the associative law to change from one parenthesization to another, the resulting concatenations of isomorphisms are always identical. It is for this strong form of associativity up to \emph{canonical} isomorphism that we may drop the parentheses without any hazard~\cite[Ch. VII]{MR1712872}. The inclined reader may consult \cite[Ch. 7]{MR633783} for further details.

\small
\bibliographystyle{abbrv}
\bibliography{ungra}

\end{document}